\documentclass[final]{mysiamltex1213}
\pdfoutput=1
\usepackage{amsfonts}
\usepackage{amsmath}
\usepackage{bm}
\usepackage{color}      
\usepackage{enumerate}
\usepackage{graphicx}
\usepackage{psfrag}
\usepackage{subfig}
\usepackage[normalem]{ulem}       

\newcommand{\stkout}[1]{\ifmmode\text{\sout{\ensuremath{#1}}}\else\sout{#1}\fi}

\newtheorem{remark}[theorem]{Remark}

\newcommand{\dist}{\mathrm{dist}}       
\newcommand{\sign}{\mathrm{sgn}}        
\newcommand{\tp}{^{T}} 
\newcommand{\dij}{\delta_{ij}}

\newcommand{\va}{\mathbf{a}}

\newcommand{\vu}{\mathbf{u}}
\newcommand{\vv}{\mathbf{v}}
\newcommand{\vw}{\mathbf{w}}
\newcommand{\vz}{\mathbf{z}}

\newcommand{\vg}{\mathbf{g}}
\newcommand{\vq}{\mathbf{q}}
\newcommand{\ve}{\mathbf{e}}
\newcommand{\vQ}{\mathbf{Q}}
\newcommand{\vr}{\mathbf{r}}

\newcommand{\vm}{\mathbf{m}}
\newcommand{\vn}{\mathbf{n}}
\newcommand{\vt}{\mathbf{t}}

\newcommand{\vx}{\mathbf{x}}
\newcommand{\vy}{\mathbf{y}}
\newcommand{\vnu}{\bm{\nu}}

\newcommand{\vzero}{\mathbf{0}}

\newcommand{\Om}{\Omega}
\newcommand{\dOm}{\partial \Omega}
\newcommand{\bdys}{\Gamma_{s}}
\newcommand{\bdyvn}{\Gamma_{\vn}}
\newcommand{\bdyvu}{\Gamma_{\vu}}

\newcommand{\iO}{\int_{\Omega}}

\newcommand{\X}{\mathbb{X}}
\newcommand{\Hbdy}[1]{H^1_{#1}(\Omega)}   
\newcommand{\Y}{V^{\perp}}   
\newcommand{\Sh}{\mathbb{S}_h}   
\newcommand{\Vh}{\mathbb{N}_h}   

\newcommand{\Yh}{\Y_h}   
\newcommand{\Uh}{\mathbb{U}_h}   
\newcommand{\R}{\mathbb{R}}   

\newcommand{\dt}{\delta t}

\newcommand{\Tk}{\mathcal{T}}

\newcommand{\Nk}{\mathcal{N}}

\title{A finite element method for nematic liquid crystals with variable degree of orientation}

\author{Ricardo H. Nochetto$^1$, ~ Shawn W. Walker$^2$, ~ Wujun Zhang$^3$}

\begin{document}
\maketitle
\slugger{mms}{xxxx}{xx}{x}{x--x}

\footnotetext[1]{\texttt{rhn@math.umd.edu}, ~$^2$\texttt{walker@math.lsu.edu}, ~$^3$\texttt{wujun@umd.edu}.}

\begin{abstract}
  We consider the simplest one-constant model, put forward by J. Ericksen, for nematic liquid crystals with variable degree of orientation.  The equilibrium state is described by a director field $\vn$ and its degree of orientation $s$, where the pair $(s, \vn)$ minimizes a sum of Frank-like energies and a double well potential.  In particular, the Euler-Lagrange equations for the minimizer contain a degenerate elliptic equation for $\vn$, which allows for line and plane defects to have finite energy.

We present a structure preserving discretization of the liquid crystal energy with piecewise linear finite elements that can handle the degenerate elliptic part without regularization, and show that it is consistent and stable. We prove $\Gamma$-convergence of discrete global minimizers to continuous ones as the mesh size goes to zero.  We develop a quasi-gradient flow scheme for computing discrete equilibrium solutions and prove it has a strictly monotone energy decreasing property. We present simulations in two and three dimensions to illustrate the method's ability to handle non-trivial defects.

A music video summary of the paper is available on YouTube: ``Mathematical Modeling and Simulation of Nematic Liquid Crystals (A Montage),''
\texttt{http://www.youtube.com/watch?v=pWWw7\_6cQ-U}.

\end{abstract}



\pagestyle{myheadings}
\thispagestyle{plain}
\markboth{R.H.~Nochetto, S.W.~Walker, W.~Zhang}{FEM For Liquid
  Crystals with Variable Degree of Orientation}

\section{Introduction}\label{sec:intro}

Complex fluids are ubiquitous in nature and industrial processes and are critical for modern engineering systems \cite{Larson_book1999, Owens_book2002, BirdHassager_book1987}.  An important difficulty in modeling and simulating complex fluids is their inherent microstructure.  Manipulating the microstructure via external forces can enable control of the mechanical, chemical, optical, or thermal properties of the material.  Liquid crystals \cite{Virga_book1994, deGennes_book1995, Calderer_SJMA2002, Ambrosio_MM1990b, Ambrosio_MM1990a, Bauman_ARMA2002, Ball_ARMA2011, Lin_CPAM1989, Lin_CPAM1991, Araki_PRL2006, Tojo_EPJE2009} are a relatively simple example of a material with microstructure that may be immersed in a fluid with a free interface \cite{Yang_JNNFM2011, Yang_JCP2013}.

Several numerical methods for liquid crystals have been proposed in
\cite{Bartels_MC2010b,  Hardt_CVPDE1988, Cohen_CPC1989, Lin_SJNA1989, Alouges_SJNA1997}
for harmonic mappings and liquid crystals with fixed degree of orientation, i.e.
a unit vector field $\vn(x)$ (called the director field) is used to represent the orientation of liquid crystal molecules.
See \cite{Guillen-Gonzalez_M2AN2013, Liu_SJNA2000, Walkington_M2AN2011} for methods that couple liquid crystals
to Stokes flow. We also refer to the survey paper \cite{Badia_ACME2011} for more numerical methods.

In this paper, we consider the one-constant model for liquid crystals with variable degree of orientation \cite{Ericksen_ARMA1991, deGennes_book1995, Virga_book1994}. The state of the liquid crystal is described by a director field $\vn(x)$ and a scalar function $s(x)$, $-1/2<s<1$, that represents the degree of alignment that molecules have with respect to $\vn$. The equilibrium state is given by $(s, \vn)$ which minimizes the so-called one-constant Ericksen's energy \eqref{energy}.

Despite the simple form of the one-constant Ericksen's model, its minimizer may have non-trivial defects.  If $s$ is a non-vanishing constant, then the energy reduces to the Oseen-Frank energy whose minimizers are harmonic maps that may exhibit point defects (depending on boundary conditions) \cite{Bethuel_book1994, Blinov_book1983, Brezis_CMP1986, Lin_CPAM1991, Lin_CPAM1989, Schoen_JDG1982}.  If $s$ is part of the minimization of \eqref{energy}, then $s$ may vanish to allow for line (and plane) defects in dimension $d=3$ \cite{Araki_PRL2006, Tojo_EPJE2009}, and the resulting Euler-Lagrange equation for $\vn$ is degenerate. However, in \cite{Lin_CPAM1991}, it was shown that both $s$ and $\vu=s\vn$ have strong limits, which enabled the study of regularity properties of minimizers and the size of defects.  This inspired the study of dynamics \cite{Calderer_SJMA2002} and corresponding numerics \cite{Barrett_M2AN2006}, which are most relevant to our paper.  However, in both cases they regularize the model to avoid the degeneracy introduced by the $s$ parameter.


We design a finite element method (FEM) \emph{without} any
regularization. We prove stability and convergence properties and
explore equilibrium configurations of liquid crystals via quasi-gradient flows.
Our method builds on \cite{Bartels_IFB2012, Bartels_SJNA2006,
  Bartels_M2AN2010} and consists of a structure preserving
discretization of \eqref{energy}. Given a weakly acute mesh
$\Tk_h$ with mesh size $h$ (see Section \ref{subsec:FEM}),
we use the subscript $h$ to denote continuous piecewise linear functions
defined over $\Tk_h$, e.g. $(s_h, \vn_h)$ is a discrete approximation of $(s, \vn)$.

Our discretization of the energy is defined in
\eqref{discrete_energy} and requires that $\Tk_h$ be \emph{weakly acute}.
This discretization preserves the underlying structure and converges to the continuous energy in the sense of $\Gamma$-convergence \cite{Braides_book2002} as $h$ goes to zero.  Next, we develop a quasi-gradient flow scheme for computing discrete equilibrium solutions.
We prove that this scheme has a strictly monotone energy decreasing property.
Finally, we carry out numerical experiments and show that our finite element method, and gradient flow, allows for computing minimizers that exhibit line and plane defects.

The paper is organized as follows.
In Section \ref{sec:Discretization}, we describe the Ericksen model
for liquid crystals with variable degree of orientation, as well as
the details of our discretization.  Section \ref{sec:consistency}
shows the $\Gamma$-convergence of our numerical method.  A
quasi-gradient flow scheme is given in Section
\ref{sec:gradient_flow}, where we also prove a strictly monotone
energy decreasing property.
Section
\ref{sec:numerics} presents simulations in two and three dimensions
that exhibit non-trivial defects in order to illustrate the method's capabilities.

\section{Discretization of Ericksen's model}\label{sec:Discretization}

We review the model \cite{Ericksen_ARMA1991} and relevant analysis
results from the literature.  We then develop our discretization
strategy and show it is stable.  The space dimension $d\ge2$
can be arbitrary.

\subsection{Ericksen's one constant model}\label{subsec:model}


Let the director field $\vn : \Omega \subset \mathbb{R}^d \rightarrow \mathbb{S}^{d-1}$ be a vector-valued function with unit length, and the degree of orientation $s: \Omega \subset \mathbb{R}^d \rightarrow [ - \frac{1}{2} , 1] $ be a real valued function.
The case $s = 1$ represents the state of perfect alignment in which all molecules are parallel to $\vn$. Likewise, $s = -1/2$ represents the state of microscopic order in which all molecules are orthogonal to the orientation $\vn$.
When $s = 0$, the molecules do not lie along any preferred direction
which represents the state of an isotropic distribution of molecules.

The equilibrium state of the liquid crystals is described by the pair $(s, \vn)$ minimizing a bulk-energy functional which in the simplest one-constant model reduces to
\vspace{-0.2cm}
\begin{equation}\label{energy}
  E[ s ,\vn]  := \underbrace{ \int_{\Omega} \Big(\kappa | \nabla s |^2 + s^2 | \nabla  \vn |^2\Big) dx }_{=:   E_1[ s ,\vn]} + \underbrace{ \int_{\Omega} \psi (s) dx }_{=:   E_2[ s ]},
\end{equation}
with $\kappa > 0$ and double well potential $\psi$, which is a $C^2$
  function defined on $-1/2 < s < 1$ that satisfies
\begin{enumerate}
  \item $\lim_{s \rightarrow 1} \psi (s) = \lim_{s \rightarrow -1/2} \psi (s) = \infty$,
  \item $\psi(0) > \psi(s^*) = \min_{s \in [-1/2, 1]} \psi(s)=0$
    for some $s^* \in (0,1)$,
    \item $\psi'(0) =0 $;
\end{enumerate}
see \cite{Ericksen_ARMA1991}.
Note that when the degree of orientation $s$ equals a non-zero
constant, the energy \eqref{energy} effectively reduces to the
Oseen-Frank energy $\int_{\Om} |\nabla \vn|^2$.
The degree of orientation $s$ relaxes the energy of
defects (i.e. discontinuities in $\vn$), which may still have finite energy
$
E[ s ,\vn]
$
if the singular set
\begin{align}\label{singularset}
  \mathcal{S} := \{ x \in \Omega ,\; s(x) = 0 \}
\end{align}
is non-empty; in this case, $\vn \notin H^1(\Omega)$.

By introducing an auxiliary variable $\vu = s \vn$
\cite{Lin_CPAM1991, Ambrosio_MM1990a}, we rewrite the energy as
\begin{equation}\label{auxiliary_energy_identity}
  E_1[ s , \vn] = \widetilde{E}_1[ s , \vu] := \int_{\Omega}
  \Big((\kappa - 1) | \nabla s |^2 + | \nabla  \vu |^2 \Big) dx,
\end{equation}
which follows from the orthogonal splitting
$\nabla \vu = \vn \otimes \nabla s + s \nabla \vn$ due to the
constraint $| \vn | = 1$.  Accordingly, we define the admissible class
\begin{equation}\label{admissibleclass}
\begin{split}
    \mathbb{A} :=& \{ (s, \vu) : \Omega \rightarrow (-1/2 , 1) \times\mathbb{R}^d:
    ~ (s,\vu) \in [H^1(\Omega)]^{d+1},\; \vu = s \vn, \vn\in \mathbb{S}^{d-1} \}.
\end{split}
\end{equation}
We say that the pair $(s,\vu)$ satisfies the {\it structural
condition} for the Ericksen energy if
\begin{equation}\label{structure}
  \vu = s \vn, ~~ -1/2 < s < 1 ~ \text{ a.e. in } \Omega,
  ~ \text{ and } \vn \in \mathbb{S}^{d-1} ~ \text{ a.e. in } \Omega.
\end{equation}
%
%
Moreover, we may enforce boundary conditions on $(s,\vu)$,
possibly on different parts of the boundary.  Let $(\bdys,\bdyvu)$
be open subsets of $\dOm$ where we set Dirichlet boundary
conditions for $(s,\vu)$.
Then we have the following restricted admissible class
\begin{align}\label{admissibleclass_BC}
  \mathbb{A}(g,\vr) := \left\{ (s, \vu) \in \mathbb{A} :
  ~ s|_{\bdys} = g, \quad \vu|_{\bdyvu} = \vr \right\},
\end{align}
for some given functions $(g,\vr)\in [W^1_\infty(\R^d)]^{d+1}$ that
satisfy the structural condition \eqref{structure} on $\partial\Omega$.
We assume the existence of $\delta_0>0$ sufficiently small such that
\begin{equation}\label{Dirichlet-restrict}
  -\frac{1}{2} + \delta_0 \le g(x), \, \vr(x)\cdot\xi \le 1 - \delta_0
  \qquad\forall \, x\in\R^d, \, \xi \in \R^d, \, |\xi|=1,
\end{equation}
and the potential $\psi$ satisfies
\begin{equation}\label{potential}
  \psi(s) \ge \psi(1-\delta_0) \quad \text{for } s \ge 1-\delta_0,
  \qquad
  \psi(s) \ge \psi(-\frac{1}{2}+\delta_0) \quad \text{for }
  s \le -\frac{1}{2}+\delta_0.
\end{equation}
This is consistent with property (1) of $\psi$.
If we further assume that
\begin{equation}\label{gne0}
g \ge \delta_0 \quad\text{ on } \partial \Omega,
\end{equation}
then the function $\vn$ is $H^1$ in a neighborhood of $\partial \Omega$ and satisfies
$\vn=g^{-1}\vr$ on $\partial \Omega$.

The existence of a minimizer $(s,\vu) \in \mathbb{A}(g,\vr)$
is shown in \cite{Lin_CPAM1991, Ambrosio_MM1990a}, but this is also a
consequence of our $\Gamma$-convergence theory.
It is worth mentioning that the constant $\kappa$ in $E[ s ,\vn]$ \eqref{energy} plays a significant role in the occurrence of defects.
Roughly speaking, if $\kappa$ is large, then $\int_{\Omega} \kappa
|\nabla s|^2 dx$ dominates the energy and $s$ is close to a
constant. In this case, defects with finite energy are less likely to
occur. But if $\kappa$ is small, then $\int_{\Omega} s^2 |\nabla
\vn|^2 dx$ dominates the energy, and $s$ may become zero. In this
case, defects are more likely to occur. (This heuristic argument is
later confirmed in the numerical experiments.) Since the investigation
of defects is of primary interest in this paper, we consider
the most significant case to be $0 < \kappa < 1$.

We now describe our finite element discretization $E_h[s_h,\vn_h]$ of the energy
\eqref{energy} and its minimizer $(s_h,\vn_h)$.

\subsection{Discretization of the energy}\label{subsec:FEM}
Let $\Tk_h =\{ T \}$ be a conforming simplicial triangulation of the domain
$\Omega$. We denote by $\Nk_h$ the set of nodes
(vertices) of $\Tk_h$ and the cardinality of $\Nk_h$
by $N$ (with some abuse of notation). We demand that $\Tk_h$ be {\it weakly acute}, namely
\begin{equation}\label{weakly-acute}
  k_{ij} := -\int_{\Om} \nabla \phi_i \cdot \nabla \phi_j dx \ge 0
  \quad\text{for all } i\ne j,
\end{equation}
where $\phi_i$ is the standard ``hat'' function associated with node
$x_i \in \Nk_h$. We indicate with $\omega_i = \text{supp} \;\phi_i$ the patch of a node $x_i$ (i.e. the ``star'' of elements in $\Tk_h$ that contain the vertex $x_i$).
Condition \eqref{weakly-acute} imposes a severe geometric restriction
on $\Tk_h$ \cite{Ciarlet_CMAME1973, Strang_FEMbook2008}. We recall the
following characterization of \eqref{weakly-acute} for $d=2$.
\smallskip
\begin{proposition}[weak acuteness in two dimensions]
\label{weak_acuteness_2D}
For any pair of triangles $T_1$, $T_2$ in $\Tk_h$ that share a common
edge $e$, let $\alpha_i$ be the angle in $T_i$ opposite to $e$ (for $i=1,2$).
If $\alpha_1 + \alpha_2 \leq \pi$ for every edge $e$, then
\eqref{weakly-acute} holds.
\end{proposition}

\smallskip\noindent
Generalizations of Proposition \ref{weak_acuteness_2D} to three dimensions, involving interior dihedral angles of tetrahedra, can be found in \cite{Korotov_MC2001, Brandts_LAA2008}.

We construct continuous piecewise affine spaces associated with the mesh, i.e.
\begin{equation}\label{eqn:discrete_spaces}
\begin{split}
  \Sh &:= \{ s_h \in H^1(\Om) : s_h |_{T} \text{ is affine for all } T \in \Tk_h \}, \\
  \Uh &:= \{ \vu_h \in H^1(\Om)^d : \vu_h |_{T} \text{ is affine in
    each component for all } T \in \Tk_h \}, \\
  \Vh &:= \{ \vn_h \in \Uh : |\vn_h(x_i)| = 1 \text{ for all nodes $x_i \in \Nk_h$} \}.
\end{split}
\end{equation}
Let $I_h$ denote the piecewise linear Lagrange interpolation
operator on mesh $\Tk_h$ with values in either $\Sh$ or $\Uh$.
We say that a pair $(s_h,\vu_h)\in\Sh\times\Uh$ satisfies the {\it discrete
structural condition} for the Ericksen energy if there exists
$\vn_h\in\Vh$ such that
\begin{equation}\label{discrete-structure}
  \vu_h = I_h[s_h \vn_h],
  \quad
  -\frac{1}{2} < s_h < 1
  \quad
  \text{in } \Omega.
\end{equation}
We then let
$g_h := I_h g$ and $\vr_h := I_h \vr$ be the discrete Dirichlet data,
and introduce the discrete spaces that include (Dirichlet) boundary conditions
\begin{equation*}\label{eqn:discrete_spaces_BC}
\begin{split}
  \Sh (\bdys,g_h) &:= \{ s_h \in \Sh : s_h |_{\bdys} = g_h \}, \quad
  \Uh (\bdyvu,\vr_h) := \{ \vu_h \in \Uh : \vu_h |_{\bdyvu} = \vr_h \},
\end{split}
\end{equation*}
as well as the discrete admissible class
\begin{equation}\label{discrete-class}
  \mathbb{A}_h(g_h,\vr_h) := \Big\{
  (s_h,\vu_h) \in \Sh (\bdys,g_h) \times \Uh (\bdyvu,\vr_h):
  (\ref{discrete-structure}) \, \textrm{holds}  \Big\}.
\end{equation}
In view of \eqref{gne0}, we can also impose the Dirichlet condition
$\vn_h = I_h[g_h^{-1} \vr_h]$ on $\dOm$.

In order to motivate our discrete version of $E_1[ s , \vn]$, note that for all $x_i \in \Nk_h$
\[
\sum_{j=1}^N k_{ij} = - \sum_{j=1}^N \int_{\Omega} \nabla \phi_i \cdot \nabla \phi_j dx = 0
\]
because $\sum_{j=1}^N \phi_j = 1$ in the domain $\Omega$; the set of
hat functions $\{\phi_j\}_{j=1}^N$ is a partition of unity.
Therefore, for piecewise linear $s_h = \sum_{i=1}^N s_h(x_i) \phi_i$, we have
\begin{align*}
  \int_{\Om} | \nabla s_h |^2 dx = -\sum_{i=1}^N k_{ii} s_h(x_i)^2 - \sum_{i, j = 1, i \neq j}^N k_{ij} s_h(x_i) s_h(x_j),
\end{align*}
whence, exploiting $k_{ii} = - \sum_{j \neq i} k_{ij}$ and the symmetry $k_{ij}=k_{ji}$, we get
\begin{equation}\label{eqn:dirichlet_integral_identity}
\begin{aligned}
  \int_{\Om} | \nabla s_h |^2 dx &= \sum_{i, j = 1}^N k_{ij} s_h(x_i) \big(s_h(x_i) - s_h(x_j)\big)
\\
&= \frac{1}{2} \sum_{i, j = 1}^N k_{ij} \big(s_h(x_i) - s_h(x_j)\big)^2 = \frac{1}{2} \sum_{i, j = 1}^N k_{ij} \big( \dij s_h \big)^2,
\end{aligned}
\end{equation}
where we define
\begin{equation}\label{eqn:delta_ij}
  \dij s_h := s_h(x_i) - s_h(x_j), \quad \dij \vn_h := \vn_h(x_i) - \vn_h(x_j).
\end{equation}
With this in mind, we define the discrete energies to be
\begin{equation}\label{discrete_energy_E1}
\begin{split}
  E_1^h [s_h, \vn_h] := & \frac{\kappa}{2} \sum_{i, j = 1}^N k_{ij} \left( \dij s_h \right)^2 + \frac{1}{2} \sum_{i, j = 1}^N k_{ij} \left(\frac{s_h(x_i)^2 + s_h(x_j)^2}{2}\right) |\dij \vn_h|^2,
\end{split}
\end{equation}
and
\begin{equation}\label{discrete_energy_E2}
  E_2^h [s_h] := \int_{\Om} \psi (s_h(x)) dx,
\end{equation}
for $(s_h,\vu_h)\in\mathbb{A}_h(g_h,\vu_h)$.
The second summation in \eqref{discrete_energy_E1} does \emph{not}
come from applying the standard discretization of $\int_{\Om} s^2 |\nabla  \vn |^2 dx$
by piecewise linear elements.  It turns out that this special form of
the discrete energy preserves the key energy inequality (Lemma \ref{lemma:energydecreasing})
which allows us to establish our $\Gamma$-convergence analysis for the degenerate coefficient $s^2$ \emph{without} regularization.
Eventually, we seek an approximation $(s_h, \vu_h) \in\mathbb{A}_h(g_h,\vr_h)$
of the pair $(s, \vu)$ such that the discrete pair $(s_h, \vn_h)$ minimizes the discrete version of the bulk energy \eqref{energy} given by
\begin{equation}\label{discrete_energy}
  E_h[s_h, \vn_h] := E_1^h [s_h, \vn_h] + E_2^h [s_h].
\end{equation}

The following result shows that definition \eqref{discrete_energy_E1}
preserves the key structure \eqref{auxiliary_energy_identity} of
\cite{Ambrosio_MM1990a,Lin_CPAM1991} at the discrete level, which turns
out to be crucial for our analysis as well.

We first introduce $\widetilde{s}_h:=I_h|s_h|$ and two
discrete versions of the vector field $\vu$
\begin{equation}\label{uh}
\vu_h:=I_h[s_h \vn_h] \in \Uh,
\quad
\widetilde\vu_h := I_h[\widetilde{s}_h \vn_h]\in\Uh.
\end{equation}
Note that both pairs $(s_h, \vu_h),(\widetilde{s}_h, \widetilde{\vu}_h)
\in \Sh\times\Uh$ satisfy \eqref{discrete-structure}.

\smallskip
\begin{lemma}[energy inequality]
\label{lemma:energydecreasing}
Let the mesh $\Tk_h$ satisfy \eqref{weakly-acute}.
If $(s_h,\vu_h)\in\mathbb{A}_h(g_h,\vr_h)$, then,
for any $\kappa > 0$, the discrete energy \eqref{discrete_energy_E1} satisfies
\begin{align}
  E_1^h [s_h, \vn_h]
  \geq
  (\kappa - 1) \iO |\nabla s_h|^2 dx + \iO |\nabla \vu_h|^2 dx
  =: \widetilde E_1^h[s_h,\vu_h],
  \label{energy_inequality}
\end{align}
as well as
\begin{align}\label{abs_inequality}
  E_1^h [s_h, \vn_h]
  \geq
  (\kappa - 1) \iO |\nabla \widetilde{s}_h|^2 dx + \iO |\nabla \widetilde{\vu}_h|^2 dx
  =: \widetilde E_1^h[\widetilde{s}_h,\widetilde\vu_h].
\end{align}
\end{lemma}

\begin{proof}
Since
\begin{align*}
  s_h(x_i) \vn_h(x_i) - s_h(x_j)\vn_h(x_j) &=
  \frac{s_h(x_i) + s_h(x_j)}{2} \big( \vn_h(x_i) - \vn_h(x_j) \big)
  \\ &+
\big(s_h(x_i) - s_h(x_j) \big) \frac{\vn_h(x_i) + \vn_h(x_j)}{2},
\end{align*}
using the orthogonality relation
$\big( \vn_h(x_i) - \vn_h(x_j)\big)\cdot \big(\vn_h(x_i) + \vn_h(x_j)\big)
= |\vn_h(x_i)|^2 - |\vn_h(x_j)|^2 = 0$ and
\eqref{eqn:dirichlet_integral_identity} yields
  \begin{align*}
    &\; \iO |\nabla \vu_h|^2 dx = \frac{1}{2} \sum_{i, j = 1}^N k_{ij} |s_h(x_i) \vn_h(x_i) - s_h(x_j)\vn_h(x_j)|^2
    \\
     = &\;
     \frac{1}{2} \sum_{i, j = 1}^N k_{ij} \left(\frac{s_h(x_i) + s_h(x_j)}{2}\right)^2 |\dij \vn_h|^2 + \frac{1}{2} \sum_{i, j = 1}^N k_{ij} (\dij s_h)^2 \left| \frac{\vn_h(x_i) + \vn_h(x_j)}{2} \right|^2.
\end{align*}
Exploiting the relations
$|\vn_h(x_i)-\vn_h(x_j)|^2 + |\vn_h(x_i)+\vn_h(x_j)|^2 = 4$ and
$\big(s_h(x_i) + s_h(x_j)\big)^2 = 2 \big(s_h(x_i)^2 +
s_h(x_j)^2\big) - \big(s_h(x_i) - s_h(x_j)\big)^2$,
we obtain
\begin{equation}\label{lem:identity}
\begin{aligned}
\iO |\nabla \vu_h|^2 dx &= \frac{1}{2} \sum_{i, j = 1}^N k_{ij}
\frac{s_h(x_i)^2 + s_h(x_j)^2}{2} |\dij \vn_h|^2
\\
&+ \frac{1}{2} \sum_{i, j = 1}^N k_{ij} (\dij s_h)^2
- \sum_{i, j = 1}^N k_{ij} (\dij s_h)^2 \left| \frac{\vn_h(x_i) - \vn_h(x_j)}{2} \right|^2,
\end{aligned}
\end{equation}
whence, we infer that
\begin{equation}\label{energyequality}
 E_1^h[s_h, \vn_h] = \int_{\Omega}
 \Big((\kappa-1) | \nabla s_h |^2 + | \nabla  \vu_h |^2 \Big) dx
 + \sum_{i, j = 1}^N k_{ij} (\dij s_h)^2 \left| \frac{\dij \vn_h}{2} \right|^2.
\end{equation}
The inequality \eqref{energy_inequality} follows directly from $k_{ij} \geq 0$ for $i \neq j$.

To prove \eqref{abs_inequality}, we note that \eqref{lem:identity}
still holds if we replace
$(s_h,\vu_h)$ with $(\widetilde{s}_h,\widetilde{\vu}_h)$:
\begin{equation}\label{energyinequality_tilde}
\begin{aligned}
\iO |\nabla \widetilde{\vu}_h|^2 dx
&= \frac{1}{2} \sum_{i, j = 1}^N k_{ij}
\frac{\widetilde{s}_h(x_i)^2 + \widetilde{s}_h(x_j)^2} 2 |\dij \vn_h|^2
\\
& + \frac{1}{2} \sum_{i, j = 1}^N k_{ij}  (\dij \widetilde{s}_h)^2
- \sum_{i, j = 1}^N k_{ij} (\dij \widetilde{s}_h)^2
\left| \frac{\vn_h(x_i) - \vn_h(x_j)}{2} \right|^2.
\end{aligned}
\end{equation}
We finally find that
\begin{align*}
\widetilde{E}_1^h[\widetilde{s}_h,\widetilde{\vu}_h]
&= \iO \Big(|\nabla \widetilde{\vu}_h|^2
+ (\kappa - 1) |\nabla \widetilde{s}_h|^2\Big) dx
= \frac{1}{2} \sum_{i, j = 1}^N k_{ij}
\frac{\widetilde{s}_h(x_i)^2 + \widetilde{s}_h(x_j)^2} 2 |\dij \vn_h|^2
\\
&
+ \frac{\kappa}{2} \sum_{i, j = 1}^N k_{ij}  (\dij \widetilde{s}_h)^2
- \sum_{i, j = 1}^N k_{ij} (\dij \widetilde{s}_h)^2 \left|
\frac{\vn_h(x_i) - \vn_h(x_j)}{2} \right|^2 \leq E_1^h[s_h, \vn_h],
\end{align*}
where we have dropped the last term and
used the triangle inequality
$|\dij \widetilde{s}_h| = \big|\widetilde{s}_h(x_i)
- \widetilde{s}_h(x_j)\big| \leq \big|s_h(x_i) -
s_h(x_j)\big| = |\dij s_h|$ along with $k_{ij}\ge0$ to obtain
\begin{equation}\label{sh-tildesh}
  \|\nabla\widetilde{s}_h\|_{L^2(\Omega)}^2
  = \frac{1}{2} \sum_{i, j = 1}^N k_{ij}  (\dij \widetilde{s}_h)^2
  \le \frac{1}{2} \sum_{i, j = 1}^N k_{ij}  (\dij {s}_h)^2
  = \|\nabla{s}_h\|_{L^2(\Omega)}^2.
\end{equation}
This concludes the proof.
\end{proof}

\smallskip
\begin{remark}[relation between \eqref{energy_inequality} and
\eqref{abs_inequality}]\label{rem:why_grad_I_h_abs_S}
Both \eqref{energy_inequality} and \eqref{abs_inequality} account
for the {\it variational crime} committed when
enforcing $\vu_h=s_h\vn_h$ and $\widetilde{\vu}_h = \widetilde{s}_h \vn_h$
\emph{only at the vertices}, and mimics \eqref{auxiliary_energy_identity}.
Since we have a precise control of the
consistency error for \eqref{energy_inequality}, this inequality will
be used later for the consistency (or lim-sup) step of
$\Gamma$-convergence of our discrete energy \eqref{discrete_energy_E1}
to the original continuous energy in \eqref{energy}. On the other hand,
\eqref{abs_inequality} has a suitable structure to prove the weak lower
semi-continuity (or lim-inf) step of $\Gamma$-convergence. This
property is not obvious when $\kappa < 1$,
the most significant case for the formation of defects.
\end{remark}

%
\section{$\Gamma$-convergence of the discrete energy}\label{sec:consistency}
%
In this section, we show that our discrete energy
\eqref{discrete_energy_E1} converges to the
continuous energy \eqref{energy} in the sense of $\Gamma$-convergence.
To this end, we first let the continuous and discrete spaces be
\[
\X := L^2(\Omega) \times [L^2(\Omega)]^d,
\qquad
\X_h := \Sh \times \Uh.
\]
We next define $E[s,\vn]$ as in \eqref{energy} for $(s,\vu)\in\mathbb{A}(g,\vr)$
and $E[s,\vu]=\infty$ for $(s,\vn)\in\X\setminus\mathbb{A}(g,\vr)$.
Likewise, we define $E_h [s_h ,\vn_h] $ as in
\eqref{discrete_energy}
for $(s_h, \vu_h)\in \mathbb{A}_h (g_h, \vr_h)$ and $E_h [s, \vn] = \infty$
for all $(s, \vu) \in \X \setminus \mathbb{A}_h (g_h, \vr_h).$

We split the proof of $\Gamma$-convergence into four subsections.
In subsection \ref{S:lim-sup}, we
use the energy $\widetilde{E}_1^h[s_h,\vu_h]$ to show the consistency
property (recall Remark \ref{rem:why_grad_I_h_abs_S}), whereas we
employ the energy $\widetilde{E}_1^h[\widetilde{s}_h,\widetilde{\vu}_h]$
in subsection \ref{S:lim-inf} to derive the weak lower semi-continuity property.
Furthermore, our functionals exhibit the usual equi-coercivity
property for both pairs $(s,\vu)$ and
$(\widetilde{s},\widetilde{\vu})$, but not for the director field
$\vn$, which is only well-defined
whenever the order parameter $s\ne0$. We discuss these issues in
subsection \ref{S:equi-coercivity} and characterize the limits
$(s,\vu)$, $(\widetilde{s},\widetilde{\vu})$ and $(s,\vn)$.
We eventually prove $\Gamma$-convergence in subsection
\ref{S:Gamma-convergence} by combining these results.

\subsection{Consistency or lim-sup property}\label{S:lim-sup}

We prove the following: if
$(s,\vu)\in\mathbb{A}(g,\vr)$, then there exists a
sequence $(s_h,\vu_h)\in\mathbb{A}_h(g_h,\vr_h)$ converging
to $(s,\vu)$ in $H^1(\Omega)$ and a discrete director
field $\vn_h\in\Vh$ converging to $\vn$ in $L^2(\Omega\setminus \mathcal{S})$ such that
\begin{equation}\label{limsup}
  E_1 [s, \vn] \geq \limsup_{h\rightarrow 0} E_1^h [s_h,\vn_h]
  \ge \limsup_{h\rightarrow 0} \widetilde{E}_1^h [s_h,\vu_h].
\end{equation}
We observe that if $(s,\vu)\notin\mathbb{A}(g,\vr)$, then
$E_1[s, \vn] = \infty$ and \eqref{limsup} is valid
for any sequences $(s_h,\vu_h)$ and $(s_h,\vn_h)$
in light of \eqref{energy_inequality}.

We first show that we can always assume
$-\frac12 + \delta_o\le s \le 1-\delta_0$ for $(s,\vu)\in\mathbb{A}(g,\vr)$.
%
%
\begin{lemma}[truncation]\label{L:truncation}
Given $(s,\vu)\in\mathbb{A}(g,\vr)$, let $(\hat{s},\hat{\vu})$ be the truncations
\[
\hat{s}(x) = \min\Big\{ 1 - \delta_0, \max \Big(-\frac{1}{2}+\delta_0,
s(x) \Big) \Big\},
\quad
\hat{\vu}(x) = \hat{s}(x) \, \vn(x)
\quad a.e. \, x\in\Omega.
\]
Then $(\hat{s},\hat{\vu}) \in \mathbb{A}(g,\vr)$ and
\[
E_1[\hat{s},\vn] \le E_1[s,\vn],
\quad
E_2[\hat{s}] \le E_2[s].
\]
The same assertion is true for any $(s_h,\vu_h) \in \mathbb{A}_h(g_h,\vr_h)$
except that the truncations are defined nodewise, i.e.
$(I_h \hat{s}_h, I_h \hat{\vu}_h)\in \mathbb{A}_h(g_h,\vr_h)$.
\end{lemma}
\begin{proof}
The fact that $(\hat{s},\hat{\vu})$ satisfy the Dirichlet boundary
conditions is a consequence of \eqref{Dirichlet-restrict}. Moreover,
$(\hat{s},\hat{\vu})\in [H^1(\Omega)]^{d+1}$ and the structural
property \eqref{structure} holds by construction, whence
$(\hat{s},\hat{\vu}) \in \mathbb{A}(g,\vr)$. We next observe that
\[
\nabla \hat{s} = \chi_{\Omega_0} \nabla s,
\qquad
\Omega_0:=\{x\in\Omega:-\frac{1}{2}+\delta_0 \le s(x) \le 1-\delta_0\};
\]
\cite[Ch. 5, Exercise 17]{Evans:book}. Consequently, we obtain
\[
E_1[\hat{s},\vn] = \int_\Omega \kappa |\nabla \hat{s}|^2
+ |\hat{s}|^2 |\nabla \vn|^2 \le
\int_\Omega \kappa |\nabla s|^2
+ |s|^2 |\nabla \vn|^2 = E_1[s,\vn],
\]
as well as
\[
E_2[\hat{s}] = \int_\Omega \psi(\hat{s}) \le \int_\Omega \psi(s) = E_2[s],
\]
because of \eqref{potential}. This concludes the proof.
\end{proof}

To construct a recovery sequence $(s_h,\vu_h)\in\mathbb{A}_h(g_h,\vr_h)$ we need
point values of $(s,\vu)$ and thus a regularization procedure
of functions in the admissible
class $\mathbb{A}(g,\vr)$. We must enforce both the structural
property \eqref{structure} and the Dirichlet boundary
conditions $s=g$ and $\vu=\vr$; neither one is guaranteed by convolution.
We are able to do this provided $\bdys=\bdyvu=\partial\Omega$
and the Dirichlet datum $g$ satisfies \eqref{gne0}.

\begin{proposition}[regularization of functions in $\mathbb{A}(g,\vr)$]
  \label{P:regularization}
Let $\bdys=\bdyvu=\partial\Omega$,
$(s,\vu) \in \mathbb{A}(g,\vr)$ and let $g$ satisfy \eqref{gne0}.
Given $\epsilon>0$ there exists a pair
$(s_\epsilon,\vu_\epsilon) \in \mathbb{A}(g,\vr)\cap [W^1_\infty(\Omega)]^{d+1}$
such that
\begin{equation}\label{H1convergence}
  \|(s,\vu) - (s_\epsilon,\vu_\epsilon)\|_{H^1(\Omega)} \le \epsilon,
\end{equation}
\begin{equation}\label{cut-off}
-\frac{1}{2} + \delta_0
\le s_\epsilon(x), \, \vu_\epsilon(x) \cdot\xi \le 1 - \delta_0
\quad\forall \, x\in\Omega, \, \xi \in \R^d, |\xi| = 1.
\end{equation}
\end{proposition}
\begin{proof}
We construct a two-scale approximation with scales $\delta<\sigma$,
which satisfies the boundary conditions exactly.
We split the argument into several steps.

\medskip

Step 1: {\it Regularization with Dirichlet condition}.
Extend $s-g\in H^1_0(\Omega)$ by zero to $\R^d \setminus \Omega$.
Let $\eta_\delta$ be a smooth and non-negative mollifier with support
contained in the ball $B_\delta(0)$ centered at $0$ with radius $\delta$.
Define $d_\delta: \R^d \to \R$ by
\[
d_\delta (x) := \chi_{\Omega}(x) \min\{ \delta^{-1}\dist(x, \partial \Omega) , 1 \},
\]
which is Lipschitz in $\R^d$, and observe that $\nabla d_\delta$
is supported in the boundary layer
\[
\omega_\delta := \{x\in\Omega: \dist(x,\partial\Omega) \le \delta \},
\]
and $|\nabla d_\delta| = \delta^{-1}\chi_{\omega_\delta}$.
We consider the Lipschitz approximations of $(s,\vu)$ given by
\[
s_\delta := d_\delta \, (s * \eta_\delta) + \big(1- d_\delta\big) g,
\qquad
\vu_\delta := d_\delta \, (\vu * \eta_\delta) + \big(1- d_\delta\big) \vr.
\]
Since $d_\delta$ vanishes on $\partial\Omega$ we readily see that
$(s_\delta,\vu_\delta)=(g,\vr)$ on  $\partial\Omega$. Moreover,
the following properties are valid
\begin{equation}\label{delta-to0}
s_\delta \to s, \quad
\vu_\delta \to \vu, \quad |\vu_\delta| \to |\vu|
\quad \text{a.e. and in } H^1(\Omega).
\end{equation}
The last property is a consequence of the middle one via triangle
inequality, and the first two are similar.
It thus suffices to show the first property for $s$. We simply write
\[
\nabla (s_\delta - s) =
\nabla d_\delta (s-g)*\eta_\delta +  \nabla d_\delta \, \big(g*\eta_\delta-g \big)
+ d_\delta \nabla(s*\eta_\delta -s)
+ \big(d_\delta-1\big) \,  \nabla (s-g).
\]
Since $s-g \in H^1(\omega_{\delta})$, and $s-g = 0$ on
$\partial \Omega$, we apply
Poincar\'e's inequality to deduce
\[
\|s-g\|_{L^2(\omega_{\delta})} \le C \delta \|\nabla (s-g) \|_{L^2(\omega_{\delta})},
\]
whence
\[
\|\nabla d_\delta (s-g)*\eta_\delta\|_{L^2(\Omega)} \le C \delta^{-1}
\|s-g\|_{L^2(\omega_{\delta})} \le C
\|\nabla (s-g) \|_{L^2(\omega_{\delta})} \to 0
~ \text{ as } \delta\to0.
\]
Likewise, a similar argument gives for the fourth term
\[
\big\|\big(d_\delta-1\big) \,  \nabla (s-g)\big\|_{L^2(\Omega)}
\le C \|\nabla (s-g) \|_{L^2(\omega_\delta)} \to 0
\quad\text{ as } \delta\to0.
\]
On the other hand, the estimate
$\|g*\eta_\delta - g\|_{L^\infty(\Omega)} \leq \delta \| \nabla g \|_{L^\infty(\R^d)}$ yields
\[
\|g*\eta_\delta - g\|_{L^2(\omega_\delta)} \le
|\omega_\delta|^{1/2} \delta \|\nabla g\|_{L^\infty(\R^d)} \le
C \delta^{\frac{3}{2}} \|\nabla g\|_{L^\infty(\R^d)},
\]
which implies, for the second term above,
\[
\big\| \nabla d_\delta \, \big(g*\eta_\delta-g \big)\big\|_{L^2(\Omega)}
\le C \delta^{\frac{1}{2}} \|\nabla g\|_{L^\infty(\R^d)}.
\]
Finally, for the third term we recall that $s\in H^1(\R^d)$ equals $g$ outside
$\Omega$ and exploit the convergence
$\nabla s*\eta_\delta \to \nabla s$ in $L^2(\Omega)$ to obtain
\[
\| d_\delta \nabla(s*\eta_\delta -s) \|_{L^2(\Omega)} \to 0
\quad\text{ as } \delta\to0.
\]

Step 2: {\it Structural condition}.
The pair $(s_\delta,\vu_\delta)$ does not satisfy the
structural condition \eqref{structure} unfortunately.
We now construct a closely related
pair that satisfies \eqref{structure}.
Recall that $(g,\vr)\in [W^1_\infty(\R^d)]^{d+1}$ satisfy
the bounds \eqref{Dirichlet-restrict} in $\R^d$, whence so do
the extensions of $(s,\vu)$ because $s=g$, $\vu=\vr$ outside $\Omega$.
Thus, we can show that
\[
-\frac{1}{2} +  \delta_0 \le s_\delta(x), \, \vu_\delta(x)\cdot\xi \le 1 - \delta_0
\quad\forall \, x\in\Omega, \, \xi \in \R^d, \, |\xi|=1;
\]
we only argue with $s_\delta$ because dealing with
$\vu_\delta\cdot\xi$ is similar. We have
$a:=-\frac{1}{2} + \delta_0 \le s*\eta \le 1 - \delta_0=:b$ because
$\eta_\delta\ge0$ and the convolution preserves constants, whence
\[
s_\delta \le d_\delta b + (1-d_\delta) b = b,
\quad
s_\delta \ge d_\delta a + (1-d_\delta) a = a
\quad\text{ in } \Omega.
\]
We next introduce the second parameter
$\sigma>\delta$ and the Lipschitz approximation of the sign function
\[
\rho_{\sigma}(t) = \min \big\{1, \max \{ -1, t/\sigma\} \big\},
\]
along with the two-scale approximation of $(s,\vu)$
\[
s_{\sigma,\delta} := \rho_\sigma(s_\delta) |\vu_\delta|,
\qquad
\vu_{\sigma,\delta} := |\rho_\sigma(s_\delta)| \vu_\delta.
\]
We note that $|s_{\sigma,\delta}| = |\vu_{\sigma,\delta}|$ by
construction, whence \eqref{structure} holds, and
$(s_{\sigma,\delta},\vu_{\sigma,\delta}) = (g,\vr)$ on
$\partial\Omega$ because $\rho_\sigma(s_\delta)=1$ on $\partial\Omega$,
for $\sigma \le \delta_0$, according to \eqref{gne0};
hence $(s_{\sigma,\delta},\vu_{\sigma,\delta})\in\mathbb{A}(g,\vr)
\cap [W^1_\infty(\Omega)]^{d+1}$. It remains to
show how to choose $\delta$ and $\sigma$, which we do next.

\medskip

Step 3: {\it Convergence in $H^1$ as $\delta\to0$}. In view of
\eqref{delta-to0} we readily deduce that
\[
s_{\sigma,\delta} \to s_\sigma:=\rho_\sigma(s) |s|,
\qquad
\vu_{\sigma,\delta} \to \vu_\sigma := |\rho_\sigma(s)| \vu
\qquad
\text{a.e. and in } L^2(\Omega).
\]
We now prove convergence also in $H^1(\Omega)$. Since
$\nabla \rho_{\sigma}(s_{\delta}) = \sigma^{-1} \chi_{\{|s_{\delta}| \leq \sigma\}}
\nabla s_{\delta}$ we get
\[
\nabla \rho_{\sigma}(s_{\delta}) - \nabla \rho_{\sigma}(s)
= \sigma^{-1} \big(\chi_{\{|s_{\delta}|\leq \sigma\}} -
\chi_{\{|s|\leq \sigma\}}\big)\nabla s
+ \sigma^{-1} \chi_{\{|s_{\delta}|\leq \sigma\}} \big(\nabla s_{\delta} - \nabla s \big).
\]
Applying the Lebesgue dominated
convergence theorem for the first term and \eqref{delta-to0} for the
second term yields, as $\delta\to0$,
\begin{equation}\label{rho}
\nabla \rho_\sigma(s_\delta) \to \nabla \rho_\sigma(s),
\qquad
\nabla |\rho_\sigma(s_\delta)| \to \nabla |\rho_\sigma(s)|
\qquad\text{ in } L^2(\Omega).
\end{equation}
The second convergence result is due to the fact that
$\nabla |f|= \sign_0(f) \nabla f$
for any $f\in W^1_1(\Omega)$, where $\sign_0(f)$ is the sign function
that vanishes at $0$ \cite[Ch. 5, Exercise 17]{Evans:book}. We next write
\begin{equation*}
\begin{aligned}
\nabla(s_{\sigma,\delta}-s_\sigma) &= \nabla \big(\rho_\sigma(s_\delta)-\rho_\sigma(s)\big)
\big(|\vu_\delta| - |\vu|  \big)
+ \nabla \rho_\sigma(s) \big(|\vu_\delta| - |\vu|  \big)
\\
& + \rho_\sigma(s_\delta) \nabla(|\vu_\delta|-|\vu|  \big)
+ \nabla\big(\rho_\sigma(s_\delta) - \rho_\sigma(s) \big) |\vu|
+ \big( \rho_\sigma(s_\delta) - \rho_\sigma(s) \big) \nabla |\vu|,
\end{aligned}
\end{equation*}
and infer that $\nabla(s_{\sigma,\delta}-s_\sigma)\to0$ as
$\delta\to0$ in $L^2(\Omega)$
upon using again the Lebesgue dominated convergence theorem for the second
and fifth terms together with \eqref{delta-to0}, \eqref{rho}, and
$|\vu|, |\vu_\delta|, |\rho_\sigma(s_\delta)| \le 1$ for the other terms.

\medskip

Step 4: {\it Convergence in $H^1$ as $\sigma\to0$}. It remains to prove
\[
s_\sigma = \rho_\sigma(s)|s| \to s,
\qquad
\vu_\sigma = |\rho_\sigma(s)| \vu \to \vu
\qquad\text{ in } H^1(\Omega).
\]
To this end, we use again that $\nabla |s| = \sign_0(s)\nabla s$ and write
\[
\nabla (s_\sigma - s) = \nabla\rho_\sigma(s) |s|
+ \big(\rho_\sigma(s) \, \sign_0(s) -1 \big) \nabla s.
\]
Since $\nabla\rho_\sigma(s) = \sigma^{-1} \chi_{\{|s|<\sigma\}}\nabla s$,
we readily obtain as $\sigma\to0$
\[
\|\nabla\rho_\sigma(s) \, |s| \, \|_{L^2(\Omega)}
\le \|\nabla s\|_{L^2(\{|s|<\sigma\})} \to 0.
\]
On the other hand, $\rho_\sigma(t)\to\sign_0(t)$ for all
$t\in\mathbb{R}$, whence
\[
\|\big(\rho_\sigma(s) \, \sign_0(s) -1 \big) \nabla s\|_{L^2(\Omega)}
\to \|\chi_{\{s=0\}} \nabla s\|_{L^2(\Omega)} = 0
\quad\textrm{ as } \sigma\to0
\]
because $\chi_{\{s=0\}} \nabla s =0$ a.e. in $\Omega$
\cite[Ch. 5, Exercise 17]{Evans:book}. Recalling that
$|\vu|=|s|$ a.e. in $\Omega$, and thus $\chi_{\{\vu=0\}} = \chi_{\{s=0\}}$,
a similar argument shows that $\vu_\sigma\to\vu$.

\medskip

Step 5: {\it Choice of $\sigma$ and $\delta$}.
Given $\epsilon>0$, we first use Step 4 to choose $\sigma$ such that
\[
\|(s,\vu) - (s_\sigma,\vu_\sigma)\|_{H^1(\Omega)} \le \frac{\epsilon}{2}.
\]
We finally resort to Step 3
to select $\delta<\sigma$, depending on $\sigma$, such that
\[
\|(s_\sigma,\vu_\sigma) - (s_{\delta,\sigma},\vu_{\delta,\sigma})\|_{H^1(\Omega)}
\le \frac{\epsilon}{2}.
\]
Therefore, we obtain the desired regularized pair, i.e. $s_\epsilon := s_{\delta,\sigma}$,
$\vu_\epsilon := \vu_{\delta,\sigma}$ satisfies
$(s_\epsilon,\vu_\epsilon) \in \mathbb{A}(g,\vr)\cap [W^1_\infty(\Omega)]^{d+1}$
along with \eqref{H1convergence} and \eqref{cut-off}. The proof is complete.
\end{proof}

We now fix $\epsilon>0$ and
let $(s_{\epsilon,h},\vu_{\epsilon,h})\in\mathbb{X}_h$
be the Lagrange interpolants of
$(s_\epsilon,\vu_\epsilon)\in\mathbb{A}(g,\vr)$
given in Proposition \ref{P:regularization}, which are well defined
because $(s_\epsilon,\vu_\epsilon)\in [W^1_\infty(\Omega)]^{d+1}$ and
satisfy $(s_{\epsilon,h},\vu_{\epsilon,h}) = (g_h,\vr_h)$ on $\partial\Omega$.
For any node $x_i$, we set
\[
  \vn_{\epsilon,h}(x_i) =
  \begin{cases}
    \vu_{\epsilon,h}(x_i) / s_{\epsilon,h}(x_i)     \quad & \text{if $s_{\epsilon,h}(x_i) \neq 0$}
    \\
    \text{any unit vector}                  \quad & \text{otherwise,}
  \end{cases}
\]
and observe that \eqref{discrete-structure} holds whence
$(s_{\epsilon,h},\vu_{\epsilon,h})\in\mathbb{A}_h(g_h,\vr_h)$.
In view of the
energy identity \eqref{energyequality}, and the property
$\|(s_{\epsilon,h},\vu_{\epsilon,h}) - (s_\epsilon,\vu_\epsilon)\|_{H^1(\Omega)}\to0$
as $h\to0$, to show \eqref{limsup} it suffices to prove that the
consistency term satisfies
\begin{align}\label{residual}
C_1^h[s_{\epsilon,h},\vn_{\epsilon,h}]
:= \sum_{i, j = 1}^N k_{ij} \big(\dij s_{\epsilon,h} \big)^2
\big| \dij \vn_{\epsilon,h} \big|^2
\rightarrow 0, \quad \text{ as $h \rightarrow 0$. }
\end{align}
Heuristically, if $\vn_\epsilon=\vu_\epsilon/s_\epsilon$ is in $W^1_{\infty}(\Omega)$, then the sum
\eqref{residual} would be
of order $h^2 \int_{\Omega} | \nabla s_{\epsilon,h} |^2 dx$, which obviously
converges to zero. However, such an argument fails if the director
field $\vn_\epsilon$ lacks high regularity, which is the case with defects.
Since $\vn_\epsilon$ is not
regular in general when $s_\epsilon$ vanishes,
the proof of consistency requires a separate
treatment of the region where $\vn_\epsilon$ is regular and the region where
$\vn_\epsilon$ is singular. The heuristic argument carries over in the
regular region, while in the singular region we appeal to basic
measure theory.
With this motivation in mind, we now prove the following lemma.
\smallskip
\begin{lemma}[lim-sup inequality]\label{lemma:limsup}
Let
$(s_\epsilon, \vu_\epsilon) \in \mathbb{A}(g,\vr) \cap [W^{1}_\infty(\Omega)]^{d+1}$ be
the functions constructed in Proposition \ref{P:regularization}, for
any $\epsilon>0$, and let
$(s_{\epsilon,h},\vu_{\epsilon,h}) \in \mathbb{A}_h(g_h,\vr_h)$
be their Lagrange interpolants. Then
\[
E_1[s_\epsilon, \vn_\epsilon] =
\lim_{h \to 0} E_1^h [s_{\epsilon,h}, \vn_{\epsilon,h}] =
\lim_{h \to 0} \widetilde{E}_1^h [s_{\epsilon,h}, \vu_{\epsilon,h}] =
\widetilde{E}_1[s_\epsilon, \vu_\epsilon].
\]
\end{lemma}
\begin{proof}
Since $\epsilon$ is fixed, we simplify the notation and write $(s_h,\vn_h)$ instead of
$(s_{\epsilon,h},\vn_{\epsilon,h})$.
In order to prove that $C_1^h[s_h,\vn_h] \to 0$ in \eqref{residual}, we choose an arbitrary
$\delta>0$ and divide the domain $\Omega$ into two disjoint regions
\[
\mathcal{S}_{\delta} := \{ x \in \Omega: \; |s_\epsilon(x)| < \delta\},
\qquad
\mathcal{K}_{\delta} := \overline{\Omega} \setminus
\mathcal{S}_{\delta},
\]
and split $C_1^h[s_h,\vn_h]$ into two parts
\begin{align*}
\mathcal{I}_h(\mathcal{K}_{\delta}) :=
\sum_{x_i, x_j \in \mathcal{K}_{\delta}}
k_{ij} \big( \dij s_h \big)^2 \big| \dij \vn_h \big|^2,
\quad
\mathcal{I}_h(\mathcal{S}_{\delta}) :=
\sum_{\text{$x_i$ or $x_j \in \mathcal{S}_{\delta}$ }}
k_{ij} \big( \dij s_h \big)^2 \big| \dij \vn_h \big|^2.
\end{align*}

{\it Step 1: Estimate on $\mathcal{K}_{\delta}$}.
Since both $s_\epsilon$ and $\vu_\epsilon$ are Lipschitz in $\overline{\Omega}$,
the set $\mathcal{K}_{\delta}$ is a compact set and the field
$\vn_\epsilon = s_\epsilon^{-1} \vu_\epsilon$ is also Lipschitz in
$\mathcal{K}_{\delta}$ with a constant that depends on
$\epsilon$ and $\delta$.
Therefore, $|\delta_{ij}\vn_h|=| \vn_h(x_i) - \vn_h(x_j) |\le C_{\epsilon,\delta} h$
because
$x_i$ and $x_j$ are connected by a single edge of the mesh,
whence
\begin{align*}
  \mathcal{I}_h(\mathcal{K}_{\delta})
\leq
C_{\epsilon,\delta} h^2 \sum_{i, j = 1 }^N k_{ij} (\dij s_h)^2
\to 0
\qquad\text{ as } h\to0,
\end{align*}
because $\frac12 \sum_{i, j = 1 }^N k_{ij} (\dij s_h)^2 = \|\nabla s_h\|_{L^2(\Omega)}^2
\le C\|\nabla s_\epsilon\|_{L^2(\Omega)}^2 < \infty$.

\medskip

{\it Step 2: Estimate on $\mathcal{S}_{\delta}$.}
If either $x_i$ or $x_j$ is in $\mathcal{S}_{\delta}$, without loss of generality, we assume that $x_i \in \mathcal{S}_{\delta}$.
Since $s_\epsilon$ is Lipschitz, and $s_h=I_hs_\epsilon$ is the Lagrange
interpolant of $s_\epsilon$, there is a mesh size $h$ such that
for any $x$ in the star $\omega_i$ of $x_i$,
$|s_h(x) - s_h(x_i)| \leq C_\epsilon h\le\delta$,
which implies that $\omega_i \subset \mathcal{S}_{2 \delta}$.
Since $|\dij \vn_h|\le 2$, we get
\begin{align*}
  \mathcal{I}_h(\mathcal{S}_{\delta}) \le
  4 \sum_{ \text{ $x_i$ or $x_j \in \mathcal{S}_{\delta}$}} k_{ij} (\dij s_h)^2
\leq
8 \int_{ \cup \, \omega_i } |\nabla s_h|^2 dx
\leq
8\int_{ \mathcal{S}_{2\delta} } |\nabla s_h|^2 dx,
\end{align*}
where the union $\cup \, \omega_i$ is taken over all nodes $x_i$ in
$\mathcal{S}_{\delta}$. If $d<p<\infty$, we infer that
\[
\int_{ \mathcal{S}_{2\delta} } |\nabla s_h|^2 dx
\le C \Big( \int_{ \mathcal{S}_{2\delta} } |\nabla I_h s_\epsilon|^p
dx \Big)^{\frac{2}{p}}
\to
C \Big( \int_{\mathcal{S}_{2\delta} } |\nabla s_\epsilon|^p dx \Big)^{\frac{2}{p}}
\quad\text{ as } h\to0,
\]
in view of the stability of the Lagrange interpolation operator $I_h$
in $W^1_p$ for $p>d$.

\medskip

{\it Step 3: The limit $\delta\to0$.}
Combining Steps 1 and 2 gives for all $\delta > 0$
\begin{equation*}
  \lim_{h \to 0} \sum_{i, j = 1}^N k_{ij} \big(\dij s_h\big)^2
  \big|\dij \vn_h \big|^2 \leq  C \Big( \int_{ \mathcal{S}_{2\delta} }
  |\nabla s_\epsilon|^p dx  \Big)^{\frac{2}{p}}
  = C \Big( \int_{\Omega } |\nabla s_\epsilon|^p \chi_{ \{|s_\epsilon| \leq
    2 \delta\}} dx  \Big)^{\frac{2}{p}}
\end{equation*}
where $\chi_A$ is the characteristic function of the set $A$.
By virtue of the Lebesgue dominated convergence theorem, we obtain
\[
 \lim_{\delta \to 0} \int_{\Omega } |\nabla s_\epsilon|^p \chi_{ \{|s_\epsilon| \leq 2 \delta\}} dx = \int_{\Omega } |\nabla s_\epsilon|^p \chi_{ \{s_\epsilon = 0\}} dx =
 0,
\]
because $\nabla s_\epsilon(x) = \vzero$ for a.e. $x\in\{s_\epsilon=0\}$
\cite[Ch. 5, Exercise 17]{Evans:book}. This proves the lemma.
\end{proof}

\subsection{Weak lower semi-continuity or lim-inf property}\label{S:lim-inf}

This property usually follows from convexity.
While it is obvious that the discrete energy $\widetilde E_1^h[s_h,\vu_h]$ in \eqref{energy_inequality} is convex with respect to $\nabla \vu_h$ and $\nabla s_h$ if $\kappa \geq 1$, the convexity is not clear if $0< \kappa < 1$.
It is worth mentioning that if $\kappa < 1$, the convexity of the continuous energy \eqref{auxiliary_energy_identity} is based on the fact that $|\vu| = |s|$ a.e. in $\Omega$ and hence the convex part $\int_{\Omega} |\nabla \vu|^2 dx$ controls the concave part $(\kappa -1) \int_{\Omega} |\nabla s|^2 dx$ \cite{Lin_CPAM1991}.
However, for the discrete energy \eqref{energy_inequality}, the
equality $|\vu_h| = |s_h|$ holds only at the vertices. Therefore, it
is not obvious how to establish the weak lower semi-continuity of
$\widetilde E_1^h[s_h,\vu_h]$.
This is why we exploit the nodal
relations $\widetilde{s}_h=|s_h| = |\vu_h| = |\widetilde{\vu}_h|$
to derive an alternative formula for $\widetilde E_1^h[\widetilde{s}_h,\widetilde\vu_h]$. Our next lemma hinges on  \eqref{abs_inequality} and makes the convexity of $\widetilde E_1^h[I_h|s_h|,\widetilde\vu_h]$ with respect to $\nabla \widetilde \vu_h$ completely explicit.
%
\begin{lemma}[weak lower semi-continuity]\label{lemma:wlsc}
The energy $\int_{\Om} L_h(\vw_h, \nabla \vw_h) dx$, with
\[
 L_h(\vw_h, \nabla \vw_h) := (\kappa - 1) | \nabla  I_h|\vw_h| |^2 + | \nabla  \vw_h |^2,
\]
is well defined for any $\vw_h \in \Uh$
and is weakly lower semi-continuous in $H^1(\Omega)$, i.e. for any
weakly convergent sequence $\vw_h \rightharpoonup \vw$ in the $H^1(\Omega)$
norm, we have
\[
  \liminf_{h \to 0} \int_{\Om} L_h(\vw_h, \nabla \vw_h) dx \geq  \int_{\Om} (\kappa -1) | \nabla |\vw| |^2 + | \nabla  \vw |^2 dx.
\]
\end{lemma}

\begin{proof}
If $\kappa \geq 1$, then the assertion follows from standard arguments.
Here, we only dwell upon $0 < \kappa < 1$ and dimension $d=2$, because
the case $d=3$ is similar.
After
extracting a subsequence (not relabeled) we can assume that $\vw_h$
converges to $\vw$ strongly in $L^2(\Omega)$ and pointwise a.e. in $\Omega$.

\smallskip
{\it Step 1: Equivalent form of $L_h$.}
We let $T$ be any triangle in the mesh $\Tk_h$, label its
three vertices as $x_0$, $x_1$, $x_2$, and define $\ve_1 := x_1 - x_0$
and $\ve_2 := x_2 - x_0$. After denoting
$\vw^i_h = \vw_h(x_i)$ for $i = 0, 1, 2$, a simple calculation yields
\begin{align*}
\nabla \vw_h = &\; (\vw^1_h - \vw^0_h) \otimes \ve_1^* + (\vw^2_h - \vw^0_h) \otimes \ve_2^*, 
\\
\nabla I_h |\vw_h| = &\; (|\vw^1_h| - |\vw^0_h|) \ve_1^* +  (|\vw^2_h| - |\vw^0_h|) \ve_2^*, 
\end{align*}
where
$\{\ve_i^*\}_{i=1}^2$ is the dual basis of $\{\ve_i\}_{i=1}^2$,
that is, $\ve_i^* \cdot \ve_j = I_{ij}$, and $I=(I_{ij})_{i,j=1}^2$ is the identity matrix.
Assuming $|\vw^i_h| + |\vw^0_h| \neq 0$, we realize that
\[
|\vw^i_h| - |\vw^0_h| = \frac{\vw^i_h + \vw^0_h}{|\vw^i_h|
  + |\vw^0_h|} \cdot (\vw^i_h - \vw^0_h).
\]
We then obtain $\nabla I_h |\vw_h| = G_h(\vw_h) : \nabla \vw_h$ where $G_h(\vw_h)$ is the $3$-tensor:
\[
G_h(\vw_h) :=  \frac{\vw^1_h + \vw^0_h}{|\vw^1_h| + |\vw^0_h|}
\otimes \ve_1 \otimes \ve_1^*
+  \frac{\vw^2_h + \vw^0_h}{|\vw^2_h| + |\vw^0_h|} \otimes \ve_2 \otimes\ve_2^*,
\quad\textrm{on } T,
\]
and the contraction between a 3-tensor and a 2-tensor in dyadic form is given by
\[
(\vg_1 \otimes \vg_2 \otimes \vg_3) : (\vm_1 \otimes \vm_2)
:= (\vg_1 \cdot \vm_1) (\vg_2 \cdot \vm_2) \vg_3.
\]
Therefore, we have
\begin{align*}
  L_h (\vw_h, \nabla \vw_h) = |\nabla \vw_h|^2 + (\kappa -1) |G_h(\vw_h) : \nabla \vw_h|^2,
\end{align*}
which expresses $L_h (\vw_h, \nabla \vw_h)$ directly in terms of
$\nabla \vw_h$ and the nodal values of $\vw_h$.

{\it Step 2: Convergence of $G_h(\vw_h)$.}
Given $\epsilon>0$, Egoroff's Theorem \cite{Wheeden_Book1977} asserts that
\[
  \vw_h \to \vw \quad \text{uniformly on $E_{\epsilon}$},
\]
for some subset $E_{\epsilon}$ and $|\Om \setminus E_{\epsilon}| \leq \epsilon$.
We now consider the set
$
A_{\epsilon}:= \{ |\vw(x)| \geq 2 \epsilon \} \cap E_{\epsilon},
$
and observe that there exists a sufficiently small $h_\epsilon$ such that for any $x \in A_{\epsilon}$
\[
 |\vw_h (x)| \geq \epsilon
 \quad
 \text{for all $h \leq h_\epsilon$.}
\]
If $G(\vw) := \frac{\vw}{|\vw|}\otimes I$, then we claim that
\begin{equation}\label{claim-step2}
  \int_{A_{\epsilon}} |G_h(\vw_h) - G(\vw)|^2 dx \to 0, ~~ \text{as } h \to 0.
\end{equation}
For any $x \in A_{\epsilon}$, let $\{T_h\}$ be a sequence of triangles
such that $x \in \overline{T_h}$. Since $|\vw_h(x)| \geq \epsilon$ and
$\vw_h$ is piecewise linear, there exists a vertex of $T_h$, which we
label as $x^0_h$, such that $|\vw^0_h| \geq \epsilon$.
To compare $G_h(\vw_h)$ with $\frac{\vw_h(x)}{|\vw_h(x)|}\otimes I$, we
use that $I = \ve_1 \otimes \ve_1^* +  \ve_2 \otimes \ve_2^*$:
\[
G_h(\vw_h) - \frac{\vw_h(x)}{|\vw_h(x)|}\otimes I = \sum_{i=1,2} \left( \frac{\vw^i_h + \vw^0_h}{|\vw^i_h| + |\vw^0_h|} - \frac{\vw_h(x)}{|\vw_h(x)|} \right) \otimes \ve_i \otimes \ve_i^*.
\]
We define
$
H(\vx, \vy) := \frac{\vx + \vy}{|\vx| + |\vy|}
$
and observe that for all $x \in A_{\epsilon}$, we have
\begin{align*}
G_h(\vw_h) - \frac{\vw_h(x)}{|\vw_h(x)|}\otimes I = \sum_{i=1,2} \left(
  H(\vw^0_h, \vw^i_h) - H(\vw_h(x), \vw_h(x))
 \right) \otimes \ve_i \otimes \ve_i^*.
\end{align*}
Next, we estimate
\begin{align*}
| H( & \vw^0_h, \vw^i_h) - H(\vw_h(x), \vw_h(x)) |
=
\left| \frac{|\vw_h(x)| (\vw^0_h + \vw^i_h) - (|\vw^0_h| + |\vw^i_h|) \vw_h(x) } {(|\vw^0_h| + |\vw^i_h|) |\vw_h(x)|} \right|
 \\
&\leq
\left| \frac{\vw^0_h + \vw^i_h - 2 \vw_h(x) } {|\vw^0_h| + |\vw^i_h|} \right|
+
\left| \frac{ (|\vw_h(x)| - |\vw^0_h|) \vw_h(x) } {(|\vw^0_h| + |\vw^i_h|) |\vw_h(x)|} \right|
+
\left| \frac{ (|\vw_h(x)| - |\vw^i_h|) \vw_h(x) } {(|\vw^0_h| + |\vw^i_h|) |\vw_h(x)|} \right|.
\end{align*}
Since $|\vw^0_h|, |\vw_h(x)| \geq \epsilon$, and $\vw_h(x) - \vw_h(x^i_h) = \nabla \vw_h \cdot (x - x^i_h)$ for all $x \in \overline{T_h}$, we have
\begin{align*}
\big| H(\vw^0_h, \vw^i_h) - H(\vw_h(x), \vw_h(x)) |
\leq C \frac{h}{\epsilon} |\nabla \vw_h|
\quad\forall x\in A_\epsilon\cap\overline{T_h}.
\end{align*}
Integrating on $A_{\epsilon}$, we obtain
\[
  \int_{A_{\epsilon}} \left|G_h(\vw_h) - \frac{\vw_h(x)}{|\vw_h(x)| } \otimes I\right|^2 dx \leq C \frac{h^2}{\epsilon^2} \int_{A_{\epsilon}} |\nabla \vw_h(x)|^2 dx
  \to 0, ~~ \text{as } h \to 0.
\]
Since $\vw_h \to \vw$ a.e. in $\Om$, and $\frac{\vw_h}{|\vw_h|} -
\frac{\vw}{|\vw|}$ is bounded, applying the dominated convergence
theorem, we infer that
\[
 \int_{A_{\epsilon}} \left| \frac{\vw_h}{|\vw_h|} - \frac{\vw}{|\vw|} \right|^2 \to 0, ~~ \text{as } h \to 0.
\]
Combining these two limits, we deduce \eqref{claim-step2}.

{\it Step 3: Convexity.} We now prove that the energy density
\[
  L(\vw, M) := |M|^2 + (\kappa -1) |G(\vw) : M|^2
\]
is convex with respect to any matrix $M$ for any vector $\vw$;
hereafter $G(\vw) = \vz\otimes I$ with $\vz=\frac{\vw}{|\vw|}$
provided $\vw \ne \vzero$ or $|\vz|\le1$ otherwise.
Note that $L(\vw, M)$ is a quadratic function of $M$, so we only need
to show that $L(\vw, M) \ge 0$ for any $M$ and $\vw$.
Thus, it suffices to show that $|G(\vw):M| \leq |M|$.

Assume that $M = \sum_{i,j} m_{ij} \vv_i \otimes \vv_j$ where
$\{\vv_i\}_{i=1}^2$ is the canonical basis on $\R^2$. Then we have
$
|M|^2 = \sum_{i, j=1}^2 m_{ij}^2
$
and a simple calculation yields
\begin{align*}
  G(\vw) : M = &\; \sum_i z_i \vv_i \otimes(\vv_1 \otimes \vv_1 + \vv_2 \otimes \vv_2) : \left(\sum_{k, l} m_{kl} \vv_k \otimes \vv_l \right)
  \\
  = &\;  \sum_{i, k, l} z_i m_{kl} \delta_{ik} \vv_l
  =  \sum_{i, l} z_i m_{il} \vv_l,
\end{align*}
where $\vz=\sum_{i=1}^2 z_i \vv_i$. Therefore, we obtain
\[
  |G(\vw):M|^2 = \sum_{ j=1}^2 \left( \sum_{i=1}^2 z_i m_{ij} \right)^2.
\]
The Cauchy-Schwarz inequality yields
\[
  \left( \sum_{i=1}^2 z_i m_{ij} \right)^2  \leq \left( \sum_{i=1}^2 z_i^2 \right) \left( \sum_{i=1}^2 m_{ij}^2 \right) \le \left( \sum_{i=1}^2 m_{ij}^2 \right),
\]
which implies
$
  |G(\vw):M|^2 \leq |M|^2
$
and $L(\vw,M)\ge0$
for any matrix $M$ and vector $\vw$.
A similar argument shows that $L_h (\vw_h, M) \ge 0$ for any
matrix $M$ and vector $\vw_h$.

{\it Step 4: Weak lower semi-continuity.}
Since $G_h(\vw_h) \to G(\vw)$ in $L^2(A_{\epsilon})$ according to
\eqref{claim-step2}, Egoroff's theorem yields
\[
  G_h(\vw_h) \to G(\vw)   \quad \text{uniformly on $B_{\epsilon}$},
\]
where $B_{\epsilon} \subset A_{\epsilon}$ and $|A_{\epsilon} \setminus B_{\epsilon}| \leq \epsilon$.
We claim that
\begin{equation}\label{claim-step4}
  \liminf_{h\to 0} \int_{\Om} L_h(\vw_h, \nabla \vw_h) dx \geq \int_{B_{\epsilon}} L(\vw, \nabla \vw) dx.
\end{equation}

Step 3 implies $L_h(\vw_h, \nabla \vw_h) \ge 0$ for all $x \in \Om$.  Hence,
\begin{align*}
  \int_{\Om} L_h(\vw_h, \nabla \vw_h) dx 
\geq &\; \int_{B_{\epsilon} } \Big(|\nabla \vw_h|^2 + (\kappa
-1)|G_h(\vw_h) : \nabla \vw_h|^2 \Big) dx .
\end{align*}
A simple calculation yields
\begin{equation*}
\int_{\Om} L_h(\vw_h, \nabla \vw_h) dx \geq \int_{B_{\epsilon} } L(\vw, \nabla \vw_h) dx
+  (\kappa -1) Q_h(\vw,\vw_h)
\end{equation*}
where
\begin{align*}
Q_h(\vw,\vw_h) := &\int_{B_{\epsilon} } \Big([(G_h(\vw_h) - G(\vw)):
  \nabla \vw_h]^t
[G_h(\vw_h): \nabla \vw_h]
\\
& + (G(\vw):\nabla \vw_h)^t [(G_h(\vw_h) - G(\vw)): \nabla \vw_h] \Big) dx.
\end{align*}
Since $L(\vw, \nabla \vw_h)$ is convex with respect to $\nabla \vw_h$ (Step 3), we have \cite[pg. 446, Sec. 8.2.2]{Evans:book}
\[
  \liminf_{h \to 0} \int_{B_{\epsilon} } L(\vw, \nabla \vw_h) dx \geq  \int_{B_{\epsilon}} L(\vw, \nabla \vw) dx .
\]

To prove \eqref{claim-step4}, it remains to show that $Q_h(\vw,\vw_h) \to 0$ as $h \to 0$.
Since $G(\vw)$ and $G_h(\vw_h)$ are bounded and $\int_{\Om} |\nabla \vw_h(x)|^2 dx$ is uniformly bounded, we have
\begin{align*}
  Q_h(\vw,\vw_h) &\leq C \int_{B_{\epsilon}} |G_h(\vw_h) - G(\vw)| |\nabla \vw_h|^2 dx
  \\
  & \le C \max_{B_{\epsilon}} \big|G_h(\vw_h) - G(\vw)\big|\int_{B_{\epsilon}}  |\nabla \vw_h|^2 dx
 \to 0 \quad\text{as } h\to0,
\end{align*}
due to the uniform convergence of $G_h(\vw_h)$ to $G(\vw)$ in $B_{\epsilon}$.
Therefore, we infer that
$
  \liminf_{h \to 0} \int_{\Om} L_h(\vw_h, \nabla \vw_h) dx \geq  \int_{B_{\epsilon} } L(\vw, \nabla \vw) dx.
$

Since the inequality above holds for arbitrarily small $\epsilon$, taking $\epsilon \to 0$ yields
\[
  \liminf_{h \to 0} \int_{\Om} L_h(\vw_h, \nabla \vw_h) dx \geq \int_{\Om \setminus \{\vw(x) = 0\}} L(\vw, \nabla \vw) dx = \int_{\Om} L(\vw, \nabla \vw) dx,
\]
where the last equality follows from $\nabla \vw = \vzero$ a.e. in the set
$\{\vw (x) = \vzero\}$
\cite[Ch. 5, exercise 17, p. 292.]{Evans:book}.
Finally, noting that $G(\vw) : \nabla \vw = \nabla |\vw|$, we get the assertion.
\end{proof}

\smallskip

\subsection{Equi-coercivity}\label{S:equi-coercivity}

We now prove uniform $H^1$-bounds for the pairs $(s_h,\vu_h)$ and
$(\widetilde{s}_h,\widetilde{\vu}_h)$, which enables us to extract
convergence subsequences in $L^2(\Omega)$ and pointwise a.e. in
$\Omega$. We then characterize and relate the limits of such sequences.

\begin{lemma}[coercivity]\label{lemma:coercivity}
For any $(s_h,\vu_h)\in\mathbb{A}_h(g_h,\vr_h)$, we have
\begin{align*}
  E_1^h [s_h, \vn_h] \geq &\; \min\{\kappa, 1\}
  \max \left\{\iO |\nabla \vu_h|^2 dx,
  \iO |\nabla s_h|^2 dx \right\}
\end{align*}
as well as
\begin{align*}
  E_1^h [s_h, \vn_h] \geq &\; \min\{\kappa, 1\}
  \max\left\{\iO |\nabla \widetilde{\vu}_h|^2 dx, \iO |\nabla I_h |s_h||^2 dx\right\}.
\end{align*}
\end{lemma}
\begin{proof}
Inequality \eqref{energy_inequality}
of Lemma \ref{lemma:energydecreasing} shows that
\begin{align}\label{eq:control_H1_norm}
  E_1^h [s_h, \vn_h]
  \geq (\kappa - 1) \iO |\nabla s_h |^2 dx + \iO |\nabla \vu_h|^2 dx.
\end{align}
If $\kappa \geq1$, then $E_1^h [s_h, \vn_h]$ obviously controls the
$H^1$-norm of $\vu_h$ with constant $1$.

If $0 < \kappa < 1$, then combining \eqref{discrete_energy_E1} with
\eqref{lem:identity} yields
\begin{align*}
  E_1^h [s_h, \vn_h] & \geq \frac{\kappa}{2} \sum_{i, j = 1}^N k_{ij}
  \left( \dij s_h \right)^2
  \\& + \frac{\kappa}{2} \sum_{i, j = 1}^N k_{ij} \left(\frac{s_h(x_i)^2 + s_h(x_j)^2}{2}\right) |\dij \vn_h|^2
 \geq \kappa \iO |\nabla \vu_h|^2 dx,
\end{align*}
whence $E_1^h [s_h, \vn_h] \geq \min\{\kappa, 1\} \iO |\nabla \vu_h|^2 dx$
as asserted. The same argument, but invoking \eqref{abs_inequality}
and \eqref{energyinequality_tilde}, leads to a similar estimate for $\int_\Omega
|\nabla\widetilde{\vu}_h|^2 dx.$

Finally, we note that \eqref{discrete_energy_E1} implies
\[
E_1^h [s_h, \vn_h] \geq \frac{\kappa}{2} \sum_{i,j=1}^N k_{ij} (\dij s_h)^2 = \kappa \int_{\Omega} |\nabla s_h|^2 dx.
\]
Upon recalling $\widetilde{s}_h = I_h |s_h|$ and noting
$|\dij s_h| \ge |\dij \widetilde{s}_h|$ and $k_{ij}\ge0$, we
deduce $\|\nabla s_h\|_{L^2(\Omega)} \ge \|\nabla\widetilde{s}_h\|_{L^2(\Omega)}$
and complete the proof.
\end{proof}

\begin{lemma}[characterizing limits]\label{lemma:char_limit}

Let $\{\Tk_h\}$ satisfy \eqref{weakly-acute} and
let a sequence $(s_h, \vu_h)\in\mathbb{A}_h(g_h,\vr_h)$ satisfy
\begin{equation}\label{uniform-bound}
E_1^h[s_h, \vn_h] \le \Lambda
\quad\text{for all } h > 0,
\end{equation}
with a constant $\Lambda>0$ independent of $h$. Then there exist subsequences
(not relabeled) $(s_h,\vu_h)\in\mathbb{X}_h$ and
$(\widetilde{s}_h,\widetilde{\vu}_h)\in\mathbb{X}_h$ weakly converging
in $[H^1(\Omega)]^{d+1}$ such that
\begin{itemize}
\item
$(s_h,\vu_h)$ converges to $(s,\vu)\in [H^1(\Omega)]^{d+1}$
in $L^2(\Omega)$ and a.e. in $\Omega$;
\item
$(\widetilde{s}_h,\widetilde{\vu}_h)$ converges to
$(\widetilde{s},\widetilde{\vu})\in [H^1(\Omega)]^{d+1}$
in $L^2(\Omega)$ and a.e. in $\Omega$;
\item
the limits satisfy $\widetilde{s}=|s|=|\vu|=|\widetilde{\vu}|$ a.e. in
$\Omega$;

\item
there exists a director field $\vn$ defined in
$\Omega$ such that $\vn_h$ converges to $\vn$ in
$L^2(\Omega\setminus\mathcal{S})$ and a.e. in $\Omega\setminus\mathcal{S}$ and
$\vu = s \vn$, $\widetilde{\vu}=\widetilde{s}\vn$ a.e. in $\Omega$.

\end{itemize}
\end{lemma}
\begin{proof}

The sequences $(s_h,\vu_h)$ and $(\widetilde{s}_h,\widetilde{\vu}_h)$
are uniformly bounded in $H^1(\Omega)$ according to Lemma
\ref{lemma:coercivity} (coercivity). Therefore, since $H^1(\Omega)$ is
compactly embedded in $L^2(\Omega)$ \cite{Adams_Book2003}, there exist
subsequences (not relabeled) that converge in $L^2(\Omega)$ and
a.e. in $\Omega$ to pairs $(s,\vu)\in [H^1(\Omega)]^{d+1}$ and
$(\widetilde{s},\widetilde{\vu})\in [H^1(\Omega)]^{d+1}$, respectively.

Since $s_h\to s$ and $\widetilde{s}_h\to\widetilde{s}$ as $h\to0$,
invoking the triangle inequality yields
\[
\big| \widetilde{s} - |s| \big| \le
\big| \widetilde{s} - \widetilde{s}_h \big|
+ \big| \widetilde{s}_h - |s_h| \big|
+ \big| |s_h| - |s| \big| \to 0
\quad \text{as } h\to0,
\]
which is a consequence of interpolation theory and \eqref{sh-tildesh}, namely
\[
| \widetilde{s}_h - |s_h| \|_{L^2(\Om)} =
\| I_h |s_h| - |s_h| \|_{L^2(\Om)} \leq C h \| \nabla |s_h| \|_{L^2(\Om)} \leq C h \|
\nabla s_h \|_{L^2(\Om)} \le C\Lambda h.
\]
A similar argument shows
\[
\| |\widetilde{\vu}_h| - I_h|\widetilde{\vu}_h| \|_{L^2(\Om)}
\leq C h \| \nabla |\widetilde{\vu}_h| \|_{L^2(\Om)}
\leq C h \|\nabla \widetilde{\vu}_h \|_{L^2(\Om)} \le C\Lambda h.
\]
Since $I_h |\widetilde{\vu}_h| = \widetilde{s}_h \to \widetilde{s}$
and $|\widetilde{\vu}_h| \to |\widetilde{\vu}|$ as $h\to0$, we deduce
$|\widetilde{\vu}| = \widetilde{s}$ a.e. in $\Omega$. Likewise, arguing instead with
the pair $(s_h,\vu_h)$ we infer that $|\vu| = |s|$ a.e. in $\Omega$.

We now define the limiting director field $\vn$ in
$\Omega\setminus\mathcal{S}$ to be $\vn = s^{-1}\vu$ and see that
$|\vn|=1$ a.e. in $\Omega\setminus\mathcal{S}$;
we define $\vn$ in $\mathcal{S}$ to be an arbitrary unit vector.
In order to relate $\vn$ with $\vn_h$, we observe that both $s_h$ and
$\vn_h$ are piecewise linear.
Applying the classical interpolation theory on each element $T$ of $\mathcal{T}_h$, we obtain
$
\| s_h \vn_h - I_h[s_h \vn_h] \|_{L^1(T)}
  \leq C h^2 \| \nabla s_h \otimes \nabla \vn_h \|_{L^1(T)}.
$
Summing over all $T \in \mathcal{T}_h$, we get
\begin{equation*}
  \| s_h \vn_h - I_h[s_h \vn_h] \|_{L^1(\Omega)}
  \leq C h^2 \| \nabla s_h \otimes \nabla \vn_h \|_{L^1(\Omega)}
  \leq C h^2 \| \nabla s_h \|_{L^2(\Omega)} \| \nabla \vn_h \|_{L^2(\Omega)}.
\end{equation*}
An inverse estimate gives $\| \nabla \vn_h \|_{L^2(\Omega)} \le Ch^{-1}$
because $|\vn_h|\le 1$. Hence
\[
\| s_h \vn_h - I_h[s_h \vn_h] \|_{L^1(\Omega)} \le C\Lambda h \to 0
\quad\text{as } h\to 0.
\]
Since $\vu_h = I_h[s_h\vn_h] \to \vu$ as $h\to0$, we discover that also
$s_h \vn_h \to  \vu$ a.e. in $\Omega$ as $h\to0$. Consequently,
for a.e. $x\in \Omega\setminus\mathcal{S}$ we have $s_h(x)\to s(x) \ne 0$
whence $s_h(x)^{-1}$ is well defined for $h$ small and
\[
\vn_h(x) = \frac{s_h(x)\vn_h(x)}{s_h(x)} \to \frac{\vu(x)}{s(x)} = \vn(x)
\quad\text{as } h\to 0.
\]
Since $|\vn_h|\le 1$, the Lebesgue dominated convergence
theorem yields
\[
\|\vn_h \chi_{\Omega\setminus\mathcal{S}} -
\vn\chi_{\Omega\setminus\mathcal{S}}\|_{L^2(\Omega)}
\to 0
\quad\text{as } h\to 0.
\]
It only remains to prove $\widetilde{\vu} = \widetilde{s}\vn$ a.e. in $\Omega$.
The same argument employed above gives
\[
\| \widetilde{s}_h \vn_h - I_h[\widetilde{s}_h \vn_h] \|_{L^1(\Omega)} \le C\Lambda h \to 0
\quad\text{as } h\to 0,
\]
whence $\widetilde{s}_h \vn_h \to \widetilde{\vu}$. This implies that
$\widetilde{s}_h(x)^{-1}$ is well defined
for a.e. $x\in\Omega\setminus\mathcal{S}$ and
\[
\vn_h(x) = \frac{\widetilde{s}_h(x)\vn_h(x)}{\widetilde{s}_h(x)}
\to \frac{\widetilde{\vu}(x)}{\widetilde{s}(x)} = \vn(x)
\quad\text{as } h\to 0.
\]
This completes the proof.
\end{proof}

\subsection{$\Gamma$-convergence}\label{S:Gamma-convergence}

We are now in the position to prove the main result, namely the convergence of global discrete minimizers.
The proof is a minor variation of the standard one
\cite{Braides_book2014,DalMaso_book1993}.

\begin{theorem}[convergence of global discrete minimizers]\label{thm:converge_numerical_soln}
Let $\{ \Tk_h \}$ satisfy \eqref{weakly-acute}. If $(s_h,\vu_h)\in\mathbb{A}_h(g_h,\vr_h)$
is a sequence of global minimizers of $E_h[s_h,\vn_h]$ in \eqref{discrete_energy}, then
every cluster point is a global minimizer of the continuous energy
$E[s ,\vn]$ in \eqref{energy}.
\end{theorem}

\begin{proof}

In view of \eqref{discrete_energy},
assume there is a constant $\Lambda>0$ such that
\[
\liminf_{h\to0} E_h[s_h,\vn_h] =
\liminf_{h\to0} \left( E_1^h[s_h,\vn_h] + E_2^h[s_h] \right) \le \Lambda,
\]
for otherwise there is nothing to prove. Applying Lemma
\ref{lemma:char_limit} yields subsequences (not relabeled)
$(\widetilde{s}_h,\widetilde{\vu}_h) \to (\widetilde{s},\widetilde{\vu})$
and $(s_h,\vu_h) \to (s,\vu)$ converging weakly in $[H^1(\Omega)]^{d+1}$,
strongly in $[L^2(\Omega)]^{d+1}$ and a.e. in $\Omega$.
Using Lemma \ref{lemma:wlsc}, we deduce
\[
\widetilde{E}_1[\widetilde{s},\widetilde{\vu}]
= \int_\Omega (\kappa -1) | \nabla \widetilde{s} |^2 + | \nabla\widetilde{\vu} |^2 dx
\le \liminf_{h\to0}
\widetilde{E}_1^h[\widetilde{s}_h,\widetilde{\vu}_h]
\le \liminf_{h\to0} E_1^h[s_h,\vn_h],
\]
where the last inequality is a consequence of \eqref{abs_inequality}.
Since $s_h$ converges a.e. in $\Omega$ to $s$, so does
$\psi(s_h)$ to $\psi(s)$. Apply now Fatou's
lemma to write
\[
E_2[s]=\int_\Omega \psi(s) = \int_\Omega \lim_{h\to0} \psi(s_h)
\le \liminf_{h\to0} \int_\Omega\psi(s_h) = \liminf_{h\to0} E_2^h[s_h].
\]
Consequently, we obtain
\[
\widetilde{E}_1[\widetilde{s},\widetilde{\vu}] + E_2[s]
\le \liminf_{h\to0} E_h[s_h,\vn_h] \le
\limsup_{h\to0} E_h[s_h,\vn_h].
\]
Moreover, the triple $(s,\vu,\vn)$ given by Lemma \ref{lemma:char_limit}
satisfy the structure property \eqref{structure}.

In view of Proposition \ref{P:regularization}, given
$\epsilon>0$ arbitrary, we can always find a pair
$(t_\epsilon,\vv_\epsilon)\in \mathbb{A}(g,\vr)\cap[W^1_\infty(\Omega)]^{d+1}$ such that
\[
\widetilde{E}_1[t_\epsilon,\vv_\epsilon] + E_2[t_\epsilon]
= E_1[t_\epsilon,\vm_\epsilon] + E_2[t_\epsilon] \le
\inf_{(t,\vm)\in\mathbb{A}(g,\vr)} E[t,\vm] + \epsilon \leq E[s,\vn] + \epsilon,
\]
where $\vm_\epsilon := t_\epsilon^{-1} \vv_\epsilon$ if
$t_\epsilon \neq 0$ or otherwise $\vm_\epsilon$ is an arbitrary unit
vector.
Apply Lemma \ref{lemma:limsup} to $(t_\epsilon,\vv_\epsilon)$ and $\vm_\epsilon$ to
find $(t_{\epsilon,h},\vv_{\epsilon,h})\in\mathbb{A}_h(g_h,\vr_h)$,
$\vm_{\epsilon,h} \in \Vh$ such that
\[
E_1[t_\epsilon,\vm_\epsilon] = \lim_{h\to0} E_1^h [t_{\epsilon,h},\vm_{\epsilon,h}].
\]
On the other hand, \eqref{potential} and
\eqref{cut-off} imply that $0\le\psi(t_{\epsilon,h})
\le \max\{\psi(-\frac{1}{2}+\delta_0),\psi(1-\delta_0)\}$
and we can invoke the Lebesgue dominated convergence theorem to infer
that
\[
E_2[t_\epsilon] = \int_\Omega \lim_{h\to0}\psi(t_{\epsilon,h})
= \lim_{h\to0} \int_\Omega \psi(t_{\epsilon,h}) = \lim_{h\to0} E_2^h[t_{\epsilon,h}].
\]
Therefore, collecting the preceding estimates, we arrive at
\[
\widetilde{E}_1[\widetilde{s},\widetilde{\vu}] + E_2[s] \le
\limsup_{h\to0} E_h[s_h,\vn_h] \le
\lim_{h\to0} E_h[t_{\epsilon,h},\vm_{\epsilon,h}]
\le E[s,\vn] + \epsilon.
\]

We now prove that $\widetilde{E}_1[\widetilde{s},\widetilde{\vu}] = E_1[s,\vn]$.
We exploit the relation $\widetilde{\vu} = \widetilde{s}\vn$ a.e. in
$\Omega$ with $|\vn|=1$, together with the fact
that $\vn$ admits a weak gradient in $\Omega\setminus\mathcal{S}$,
to find the orthogonal decomposition
$\nabla \widetilde \vu = \nabla\widetilde{s}\otimes\vn +
\widetilde{s} \, \nabla \vn$ a.e. in $\Om\setminus\mathcal{S}$. Hence
\begin{align*}
\widetilde{E}_1[\widetilde{s},\widetilde{\vu}] =
\int_{\Om\setminus\mathcal{S}} (\kappa -1) | \nabla \widetilde{s} |^2
+ | \nabla  \widetilde{\vu} |^2 dx
= &\;
\int_{\Om\setminus\mathcal{S}} \kappa | \nabla \widetilde{s} |^2
+ \widetilde{s}^2 | \nabla \vn |^2 dx
\\
= &\;
\int_{\Om\setminus\mathcal{S}} \kappa | \nabla s |^2 + s^2 | \nabla \vn |^2 dx
\equiv E_1[s, \vn]
\end{align*}
because $\widetilde{s} = |s|$ and $\| \nabla |s| \|_{L^2(\Om\setminus\mathcal{S})} =
\| \nabla s \|_{L^2(\Om\setminus\mathcal{S})}$.  Note that the
singular set $\mathcal{S}$ does
not contribute because $\| \nabla s \|_{L^2(\mathcal{S})} = \| s \nabla \vn \|_{L^2(\mathcal{S})} = 0$.
Finally, letting $\epsilon \to 0$, we see that the pair $(s,\vn)$ is a
global minimizer of $E$ as asserted.
\end{proof}

If the global minimizer of the continuous energy $E[s ,\vn]$ is unique,
then Theorem \ref{thm:converge_numerical_soln}
readily implies that the discrete energy minimizer
$(s_h, \vn_h)$ converges to the unique minimizer of $E[s ,\vn]$.
This theorem is about global minimizers only, both discrete and continuous.
In the next section, we design a quasi-gradient flow to compute
discrete local minimizers, and show its convergence (see Theorem
\ref{energydecreasing}). In general, convergence to a global
minimizer is not available, nor are rates of convergence due to the
lack of continuous dependence results. However, if local minimizers of
$E[s,\vn]$ are isolated, then there exists local minimizers of
$E_h[s_h,\vn_h]$ that $\Gamma$-converge to $(s,\vn)$
\cite{Braides_book2014,DalMaso_book1993}.

\section{Quasi-Gradient Flow}\label{sec:gradient_flow}

We consider a gradient flow methodology consisting of a gradient
flow in $s$ and a minimization in $\vn$ as a way to compute
minimizers of \eqref{energy} and \eqref{discrete_energy}.  We begin
with its description for the continuous system and verify that it has a monotone energy decreasing property.  We then do the same for the discrete system.

\subsection{Continuous case}

We introduce the following subspace to enforce Dirichlet boundary
conditions on open subsets $\Gamma$ of $\dOm$:
\begin{equation}\label{eqn:dirichlet_BC}
  \Hbdy{\Gamma} = \{ v \in H^1(\Om) : v = 0  \text{ on } \Gamma \}.
\end{equation}
Let the sets $\bdys, \bdyvn$ satisfy $\bdyvn =
\bdyvu \subset \bdys \subset\partial\Omega$ and \eqref{gne0} be
valid on $\bdys$. Therefore, the traces $\vn=\vq:=g^{-1}\vr$ and
$\vn_h=\vq_h:=I_h[g_h^{-1}\vr_h]$ are well defined on $\bdyvn$.

\subsubsection{First order variation}

Consider the bulk energy $E[s ,\vn]$ where the pair $(s, \vu)$,
with $\vu=s\vn$, is in
the admissible class $\mathbb{A} (g,\vr)$ defined in
\eqref{admissibleclass}. We take a variation $z\in H^1_0(\Omega)$ of $s$ and obtain
$\delta_s E [s ,\vn ; z] = \delta_s E_1 [ s ,\vn ; z] + \delta_s E_2 [ s ; z]$,
the first variation of $E$ in the direction $z$, where
\begin{align*}
\delta_s E_1 [s ,\vn ; z] = 2 \int_{\Omega} ( \nabla s \cdot \nabla z
+ |\nabla \vn|^2 s z ) \, dx
\quad \text{and} \quad
\delta_s E_2[ s; z] =  \int_{\Omega} \psi'(s) z \, dx.
\end{align*}
Next, we introduce the space of tangential variations of $\vn$:
\begin{align}\label{eqn:tangent_variation_space}
    \Y (\vn) &= \left\{ \vv \in H^1(\Om)^d : \vv \cdot \vn = 0 \text{ a.e. in } \Omega \right\}.
\end{align}
In order to satisfy the constraint $|\vn|=1$, we take a variation $\vv\in\Y(\vn)$ of $\vn$ and get
\begin{align*}
    \delta_{\vn} E [s ,\vn ; \vv] = \delta_{\vn} E_1 [ s ,\vn ; \vv] =
    2 \iO s^2 (\nabla \vn \cdot \nabla \vv ) \, dx.
\end{align*}
Note that variations in $\Y(\vn)$ preserve the unit length constraint
up to second order accuracy \cite{Virga_book1994}: $|\vn + t \vv|^2 = 1 + t^2 |\vv|^2$ and
$|\vn + t \vv| \ge 1$ for all $t\in\R$.

\subsubsection{Quasi-gradient flow}

We consider an $L^2$-gradient flow for $E$ with respect to the scalar
variable $s$:
\begin{align*}
\int_{\Omega} \partial_t s z \, dx := -\delta_s E_1 [ s ,\vn ; z] - \delta_s E_2 [ s ; z]
\quad\text{for all } z \in \Hbdy\bdys;
\end{align*}
here, we enforce stationary Dirichlet boundary conditions for $s$ on
the set $\bdys \subset \dOm$, whence $z=0$ on $\bdys$.
A simple but formal integration by parts yields
\begin{align*}
\int_{\Omega}  \partial_t s z \, dx = \;- \int_{\Omega}  \big( -2\Delta s +
2|\nabla \vn|^2 s + \psi'(s) \big) z \, dx
\quad \text{ for all } z \in \Hbdy{\bdys},
\end{align*}
where we use the implicit Neumann condition $\vnu \cdot \nabla s = 0$
on $\dOm \setminus \bdys$, $\vnu$ being the outer unit normal on $\dOm$.
Therefore, $s$ satisfies the (nonlinear) parabolic PDE:
\begin{align}\label{pde}
\partial_t s = 2\Delta s - 2|\nabla \vn|^2 s - \psi'(s).
\end{align}

Given $s$ satisfying \eqref{gne0} on $\bdys$,
let $\vn$ satisfy $|\vn|=1$ a.e. in $\Omega$, the stationary
Dirichlet boundary condition $\vn=\vq$ on the open set $\bdyvn \subset \dOm$,
and the following degenerate minimization problem:
\begin{equation*}
  E[s,\vn] \le E[s,\vm]
  \quad\text{for all } |\vm|=1 \text{ a.e. } \Omega,
\end{equation*}
with the same boundary condition as $\vn$. This implies
\begin{equation}\label{deg-min}
  \delta_\vn E[s,\vn;\vv] = 0
  \quad \text{for all } \vv \in \Y(\vn) \cap \Hbdy{\bdyvn}^d.
\end{equation}

\subsubsection{Formal energy decreasing property}

Differentiating the energy with respect to time, we obtain
\begin{align*}
  \partial_t E [s ,\vn ] =
  \delta_s E[s,\vn;\partial_t s] + \delta_\vn E[s,\vn;\partial_t\vn].
\end{align*}
By virtue of \eqref{pde} and \eqref{deg-min}, we deduce that
\begin{align}\label{energy-decrease}
\partial_t E [s ,\vn ] =
- \delta_s E[s,\vn;\partial_t s] = - \int_\Omega |\partial_t s|^2 \, dx.
\end{align}
Hence, the bulk energy $E$ is monotonically decreasing for our quasi-gradient flow.

\subsection{Discrete case}

Let $s_h^k\in\Sh (\bdys,g_h)$ and $\vn_h^k\in\Vh (\bdyvn,\vq_h)$
denote finite element functions with Dirichlet conditions
$s_h^k=g_h$ on $\bdys$ and $\vn_h^k = \vq_h$ on $\bdyvn$,
where $k$ indicates a ``time-step'' index (see Section \ref{algorithm} for the discrete gradient flow algorithm).
To simplify notation, we use the following:
\[
s_i^k := s_h^k(x_i),
\quad
\vn_i^k := \vn_h^k(x_i),
\quad
z_i := z_h(x_i),
\quad
\vv_i := \vv_h(x_i).
\]

\subsubsection{First order variation}

First, we introduce the discrete version of \eqref{eqn:tangent_variation_space}:
\begin{equation}\label{eqn:discrete_tangent_variation_space}
\begin{split}
  \Yh (\vn_h) &= \{ \vv_h \in \Uh : \vv_h(x_i) \cdot \vn_h(x_i) = 0
  \text{ for all nodes } x_i \in \Nk_h \}.
\end{split}
\end{equation}
Next, the first order variation of $E^h_1$ in the direction
  $\vv_h \in \Yh (\vn_h^k) \cap \Hbdy{\bdyvn}$ at the director variable
  $\vn_h^k$ reads
\begin{equation}\label{first_order_discrete_variation_dvN}
\begin{split}
 \delta_{\vn_h} E^h_1 [s_h^k, \vn_h^k ; \vv_h] &=
 \sum_{i, j = 1}^N k_{ij} \left(\frac{(s_i^k)^2 + (s_j^k)^2 }{2}\right) ( \dij \vn_h^k ) \cdot ( \dij \vv_h ),
\end{split}
\end{equation}
whereas the first order variation of $E^h_1$ in the direction
$z_h \in \Sh \cap \Hbdy{\bdys}$ at the degree of orientation variable $s_h^k$
consists of two terms
\begin{equation}\label{first_order_discrete_variation_dS}
\begin{split}
 \delta_{s_h} E^h_1 [s_h^k ,\vn_h^{k} ; z_h] &=
 \kappa \sum_{i, j = 1}^N k_{ij} \left( \dij s_h^k \right) \left( \dij z_h \right)
 + \sum_{i, j = 1}^N k_{ij}  |\dij \vn_h^{k}|^2 \left(\frac{s_i^k z_i+ s_j^k z_j}{2}\right) .
\end{split}
\end{equation}

To design an unconditionally stable scheme for the discrete gradient
flow of $E_2^h[s_h]$, we employ the convex splitting technique in
\cite{Wise_SJNA2009, Shen_DCDS2010, Shen_SJSC2010}.
We split the double well potential into a convex and concave part:
let $\psi_c $ and $\psi_e$ be both convex for all $s \in (-1/2, 1)$
so that $\psi (s) = \psi_c(s) - \psi_e(s)$,
and set
\begin{align}\label{variationE2}
 \delta_{s_h} E^h_2 [ s_h^{k+1}; z_h ] := \int_{\Omega} \big[ \psi_c'(s_h^{k+1}) - \psi_e'(s_h^{k}) \big] z_h dx.
\end{align}

\begin{lemma}[convex-concave splitting]\label{E2}
For any $s_h^{k}$ and $s_h^{k+1}$ in $\Sh$, we have
\[
\iO \psi(s_h^{k+1}) dx - \iO \psi(s_h^{k})  dx  \leq \delta_{s_h} E^h_2 [ s_h^{k+1}; s_h^{k+1} - s_h^k ] .
\]
\end{lemma}

\begin{proof}
A simple calculation, based on the mean-value theorem and the
convex splitting $\psi=\psi_c-\psi_e$,  yields
\begin{align*}
\int_{\Omega} \big( \psi(s_h^{k+1}) - \psi(s_h^{k}) \big) dx
=
 \delta_{s_h} E^h_2 [ s_h^{k+1}; s_h^{k+1} - s_h^k ] + T,
\end{align*}
where
\begin{align*}
T = &\; \int_{\Omega} \int_0^1 \big[ \psi_c'( s_h^{k} +  \theta
  (s_h^{k+1} - s_h^{k} ) ) - \psi_c'(s_h^{k+1}) \big] (s_h^{k+1} -
s_h^k) \, d \theta \, dx
\\
+ &\;
\int_{\Omega} \int_0^1 \big[ \psi_e'( s_h^{k} ) - \psi_e'( s_h^{k} +
  \theta (s_h^{k+1} - s_h^{k}  ) ) \big] (s_h^{k+1} - s_h^k) \, d
\theta \, dx .
\end{align*}
The convexity of both $\psi_c$ and $\psi_e$ implies $T \leq 0$, as desired.
\end{proof}

\subsubsection{Discrete quasi-gradient flow algorithm}\label{algorithm}

Our scheme for minimizing the discrete energy $E_h [ s_h ,\vn_h]$ is
an alternating direction method, which minimizes with respect to
$\vn_h$ and evolves $s_h$ separately in the steepest descent
direction during each
iteration. Therefore, this algorithm is not a standard gradient flow.

\medskip\noindent
{\bf Algorithm} (discrete quasi-gradient flow):
Given $(s_h^{0}, \vn_h^0)$ in $\Sh (\bdys, g_h) \times \Vh (\bdyvn, \vq_h)$,
iterate Steps (a)-(c) for $k\ge0$.

\smallskip\noindent
{\bf Step (a): Minimization.} Find $\vt_h^k\in\Yh (\vn_h^k) \cap \Hbdy{\bdyvn}$ such that $\vn_h^k + \vt_h^k$ minimizes the energy
$
E^h_1 [ s_h^k ,\vn_h^k + \vv_h]
$
for all $\vv_h$ in $\Yh (\vn_h^k) \cap \Hbdy{\bdyvn}$, i.e. $\vt_h^k$ satisfies
\begin{align*}
\delta_{\vn_h} E^h_1 [s_h^k ,\vn_h^k + \vt_h^k; \vv_h] = 0, ~\forall \vv_h \in \Yh (\vn_h^k) \cap \Hbdy{\bdyvn}.
\end{align*}
{\bf Step (b): Projection.} Normalize
$\vn_i^{k+1} := \frac{ \vn_i^k +   \vt_i^k } { | \vn_i^k + \vt_i^k  | }$ at all nodes $x_i \in \Nk_h$.

\noindent
{\bf Step (c): Gradient flow.} Using $(s_h^k, \vn_h^{k+1})$, find $s_h^{k+1}$ in $\Sh (\bdys, g_h)$ such that
\begin{align*}
\int_{\Omega}  \frac{s_h^{k+1} - s_h^{k}}{\delta t} z_h dx= -
\delta_{s_h} E^h_1 [s_h^{k+1} ,\vn_h^{k+1} ; z_h] - \delta_{s_h} E^h_2 [s_h^{k+1} ; z_h]
\quad\forall z_h \in \Sh \cap \Hbdy{\bdys}.
\end{align*}
We impose Dirichlet boundary conditions to both $s_h^k$ and $\vn_h^k$. Note that the scheme has no restriction on the time step thanks to the implicit Euler method in Step (c).

\subsection{Energy decreasing property}

The quasi-gradient flow scheme in Section \ref{algorithm} has a
monotone energy decreasing property, a discrete version of
\eqref{energy-decrease}, provided the mesh $\Tk_h$ is weakly acute, namely it
satisfies \eqref{weakly-acute} \cite{Ciarlet_CMAME1973, Strang_FEMbook2008}.

\smallskip
\begin{theorem}[energy decrease]
  \label{energydecreasing}
Let $\Tk_h$ satisfy \eqref{weakly-acute}. The iterate
$(s_h^{k+1}, \vn_h^{k+1})$ of the Algorithm (discrete quasi-gradient
flow) of Section \ref{algorithm} exists and satisfies
  \[
    E^h [s_h^{k+1} ,\vn_h^{k+1}] \leq E^h [s_h^k ,\vn_h^k ] - \frac{1}{\delta t} \int_{\Omega} (s_h^{k+1} - s_h^k)^2 dx.
  \]
  Equality holds if and only if $(s_h^{k+1}, \vn_h^{k+1}) = (s_h^k,\vn_h^k)$
  (equilibrium state).
\end{theorem}

\begin{proof}
The Steps (a) and (b) are monotone whereas Step (c) decreases the energy.

{\it Step (a): Minimization.}
Since $E_1^h$ is convex in $\vn_h^k$ for fixed $s_h^k$, there
exists a tangential variation $\vt_h^k$ which minimizes
$E_1^h [s_h^k ,\vn_h^k + \vv_h^k]$ among all tangential variations $\vv_h^k$.
The fact that $E_2^h$ is independent of the director field $\vn_h^k$ implies
\begin{align*}
 E^h [s_h^k ,\vn_h^k + \vt_h^k] \leq E^h [s_h^k ,\vn_h^k ].
\end{align*}

{\it Step (b): Projection.} Since the mesh $\Tk_h$ is weakly acute, we claim that
\begin{align*}
  \vn_h^{k+1}=\frac{\vn_h^k + \vt_h^k}{|\vn_h^k + \vt_h^k|}
  \quad\Rightarrow\quad
  E_1^h \big[s_h^k , \vn_h^{k+1} \big] \leq E_1^h \big[s_h^k , \vn_h^k + \vt_h^k \big].
\end{align*}
We follow \cite{Alouges_SJNA1997, Bartels_SJNA2006}.
Let $\vv_h = \vn_h^k + \vt_h^k$, $\vw_h = \frac{\vv_h}{|\vv_h|}$,
and observe that $|\vv_h|\ge1$ and $\vw_h$ is well-defined.
By \eqref{discrete_energy_E1} (definition of discrete energy), we only need to show that
\begin{align*}
  k_{ij} \frac{(s_i^k)^2 + (s_j^k)^2 }{2} \big|\vw_h(x_i) - \vw_h(x_j)\big|^2 \leq  k_{ij} \frac{(s_i^k)^2 + (s_j^k)^2 }{2} \big|\vv_h(x_i) - \vv_h(x_j)\big|^2 .
\end{align*}
for all $x_i,x_j\in\mathcal{N}_h$. Because $k_{ij}\ge0$ for $i\ne j$, this is equivalent to showing that
$ |\vw_h(x_i) - \vw_h(x_j)| \leq |\vv_h(x_i) - \vv_h(x_j)|$. This
follows from the fact that the mapping $\va \mapsto \va/|\va|$ defined
on $\{\va \in \mathbb R^d: |\va| \geq 1 \}$
is Lipschitz continuous with constant $1$.
Note that equality above holds if and only if $\vn_h^{k+1} = \vn_h^{k}$
or equivalently $\vt_h^k = \vzero$.

{\it Step (c): Gradient flow.}
Since $E_1^h$ is quadratic in terms of $s_h^k$, and
\[
2 s_h^{k+1} \big(s_h^{k+1} - s_h^k \big)
= \big(s_h^{k+1} - s_h^k \big)^2
+ \big|s_h^{k+1}\big|^2 - \big|s_h^k\big|^2,
\]
reordering terms gives
\begin{align*}
 E_1^h [ s_h^{k+1} ,\vn_h^{k+1} ] - E_1^h [ s_h^k ,\vn_h^{k+1} ]
= &\;
R_1
-  E^h_1 [ s_h^{k+1} -s_h^k ,\vn_h^{k+1}]
\leq R_1,
\end{align*}
where
\begin{equation*}
R_1 :=  \delta_{s_h} E^h_1 [ s_h^{k+1} ,\vn_h^{k+1} ; s_h^{k+1} - s_h^k].
\end{equation*}
On the other hand, Lemma \ref{E2} implies
\begin{align*}
E_2^h [ s_h^{k+1} ] - E_2^h [ s_h^k ]
=
\iO  \psi(s_h^{k+1}) dx  - \iO \psi(s_h^{k}) dx \leq R_2:= \delta_{s_h} E^h_2 [ s_h^{k+1} ; s_h^{k+1} - s_h^k].
\end{align*}
Combining both estimates and invoking Step (c) of the Algorithm yields
\[
E^h [ s_h^{k+1} ,\vn_h^{k+1} ] - E^h [ s_h^k ,\vn_h^{k+1} ] \leq R_1 + R_2
= - \frac {1}{\delta t} \int_{\Omega} (s_h^{k+1} - s_h^k)^2 \, dx \leq 0,
\]
which is the assertion.
Note finally that equality occurs if and only if $s_h^{k+1} = s_h^k$
and $\vn_h^{k+1} = \vn_h^k$, which corresponds to an equilibrium state.
This completes the proof.
\end{proof}

\section{Numerical experiments}\label{sec:numerics}

We present computational experiments to illustrate our method, which was implemented with the MATLAB/C++ finite element toolbox FELICITY \cite{FELICITY_REF}.  For all 3-D simulations, we used the algebraic multi-grid solver (AGMG) \cite{Notay_ETNA2010, Napov_NLAA2011, Napov_SISC2012, Notay_SISC2012} to solve the linear systems in parts (a) and (c) of the quasi-gradient flow algorithm.  In 2-D, we simply used the ``backslash'' command in MATLAB.

\subsection{Tangential variations}

Solving step (a) of the Algorithm requires a tangential basis for the test function and the solution.  However, forming the matrix system is easily done by first ignoring the tangential variation constraint (i.e. arbitrary variations), followed by a simple modification of the matrix system.

Let $A \vt_h^{k} = B$ represent the linear system in Step (a) and suppose $d=3$.  Multiplying by a discrete test function $\vv_h$, we have
\begin{equation*}
  \vv_h\tp A \vt_h^{k} = \vv_h\tp B, \quad \text{for all } \vv_h \in \R^{d N}.
\end{equation*}
Next, using $\vn_h^{k}$, find $\vr_1$, $\vr_2$ such that $\{ \vn_h^{k}, \vr_1, \vr_2 \}$ forms an orthonormal basis of $\R^3$ at each node $x_i$, i.e. find an orthonormal basis of $\Yh (\vn_h^{k})$.  Next, expand $\vt_h^{k} = \Phi_1 \vr_1 + \Phi_2 \vr_2$ and make a similar expansion for $\vv_h$.  After a simple rearrangement and partitioning of the linear system, one finds it decouples into two smaller systems: one for $\Phi_1$ and one for $\Phi_2$.  After solving for $\Phi_1$, $\Phi_2$, define the nodal values of $\vt_h^{k}$ by the formula $\vt_h^{k} = \Phi_1 \vr_1 + \Phi_2 \vr_2$.

\subsection{Point defect in 2-D}\label{sec:point_defect_2D}

For the classic Frank energy $\int_{\Om} |\nabla \vn|^2$, a point defect in two dimensions has infinite energy \cite{Virga_book1994}.  This is not the case for the energy \eqref{energy}, because $s$ can go to zero at the location of the point defect, so the term $\int_{\Om} s^2 |\nabla \vn|^2$ remains finite.
\begin{figure}
\begin{center}
  \includegraphics[width=4.0in]{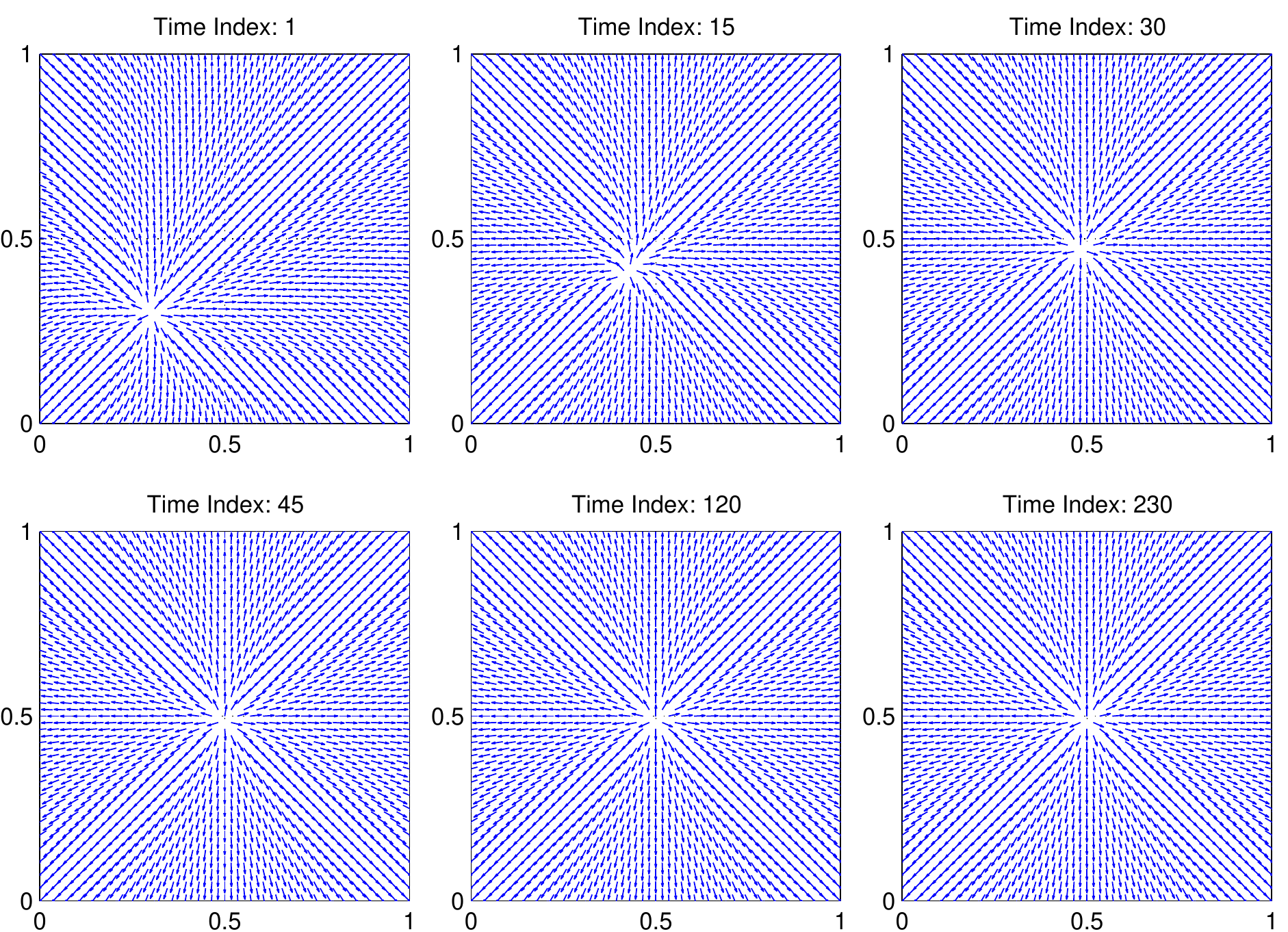}
  \includegraphics[width=4.0in]{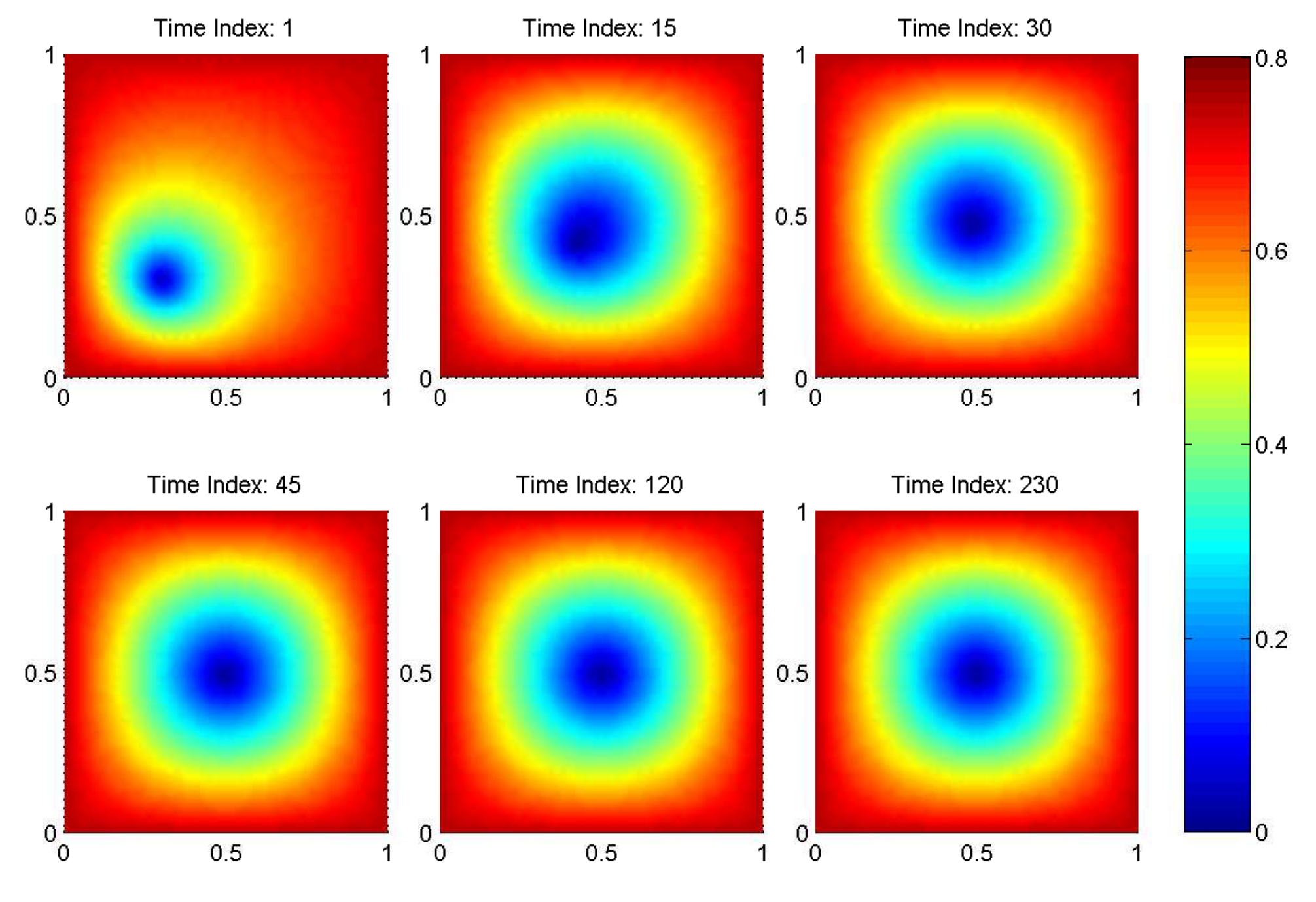}
\caption{Evolution of a point defect toward its equilibrium state (Section \ref{sec:point_defect_2D}).  Time step is $\dt = 0.02$.  The minimum value of $s$, at time index 230, is $2.0226 \cdot 10^{-2}$.}
\label{fig:Point_Defect_2D_Director_Scalar}
\end{center}
\end{figure}

We simulate the gradient flow evolution of a point defect moving to the center of the domain ($\Om$ is the unit square).  We set $\kappa=2$ and take the double well potential to have the following splitting:
\begin{equation*}
\begin{split}
  \psi (s) &= \psi_c (s) - \psi_e (s) \\
  &= 63.0 s^2 - (-16.0 s^4 + 21.33333333333 s^3 + 57.0 s^2),
\end{split}
\end{equation*}
with local minimum at $s=0$ and global minimum at $s=s^* :=
0.750025$ (see Section \ref{subsec:model} and note that a
vertical shift makes $\psi(s^*)=0$ without affecting the gradient flow).
We impose the following Dirichlet boundary conditions for $s$ and $\vn$
\begin{equation}\label{eqn:bc_point_defect}
  s = s^*, \qquad \vn = \frac{(x,y) - (0.5,0.5)}{|(x,y) - (0.5,0.5)|},
\end{equation}
on $\bdys=\bdyvn=\partial \Om$.
Initial conditions on $\Om$ for the gradient flow are: $s = s^*$ and a regularized point defect away from the center.

Figure \ref{fig:Point_Defect_2D_Director_Scalar} shows the evolution of the director field $\vn$ and the scalar degree of orientation parameter $s$.  One can see the regularizing effect that $s$ has.  We note that an $L^2$ gradient flow scheme, instead of the quasi (weighted) gradient flow we use, yields a much slower evolution to equilibrium.

\subsection{Plane defect in 3-D}\label{sec:plane_defect_3D}

Next, we simulate the gradient flow evolution of the liquid crystal
director field toward a \emph{plane} defect in the unit cube
  $\Om = (0, 1)^3$. This is motivated by an exact solution found in
\cite[Sec. 6.4]{Virga_book1994}. We set $\kappa=0.2$ and \emph{remove}
the double well potential.
We impose mixed boundary conditions for $(s,\vn)$, with
Dirichlet conditions on $\bdys=\bdyvn=\overline{\partial \Om} \cap ( \{ z = 0\} \cup \{ z = 1\} )$
\begin{equation}\label{eqn:bc_plane_defect}
\begin{split}
  z=0:& \quad s = s^*, \qquad \vn = (1,0,0), \\
  z=1:& \quad s = s^*, \qquad \vn = (0,1,0),
\end{split}
\end{equation}
and Neumann conditions $\vnu \cdot \nabla s = 0$ and $\vnu \cdot \nabla
\vn = 0$ on the remaining part of
$\partial \Om$; these conditions are not covered by
Section \ref{sec:consistency} but we explore them computationally.
The exact solution $(s,\vn)$ (at
equilibrium) only depends on $z$ and is given by
\begin{equation}\label{eqn:bc_plane_defect_exact_soln}
\begin{split}
  \vn(z) &= (1,0,0), ~ \text{for } z < 0.5, \quad  \vn(z) = (0,1,0), ~ \text{for } z > 0.5, \\
  s(z) &= 0, ~\text{at } z = 0.5, \text{ and } s(z) \text{ is linear for } z \in (0, 0.5) \cup (0.5, 1.0).
\end{split}
\end{equation}
Initial conditions on $\Om$ for the gradient flow are: $s = s^*$ and a regularized point defect away from the center of the cube.

\begin{figure}[ht]
\begin{center}


\includegraphics[width=4.5in]{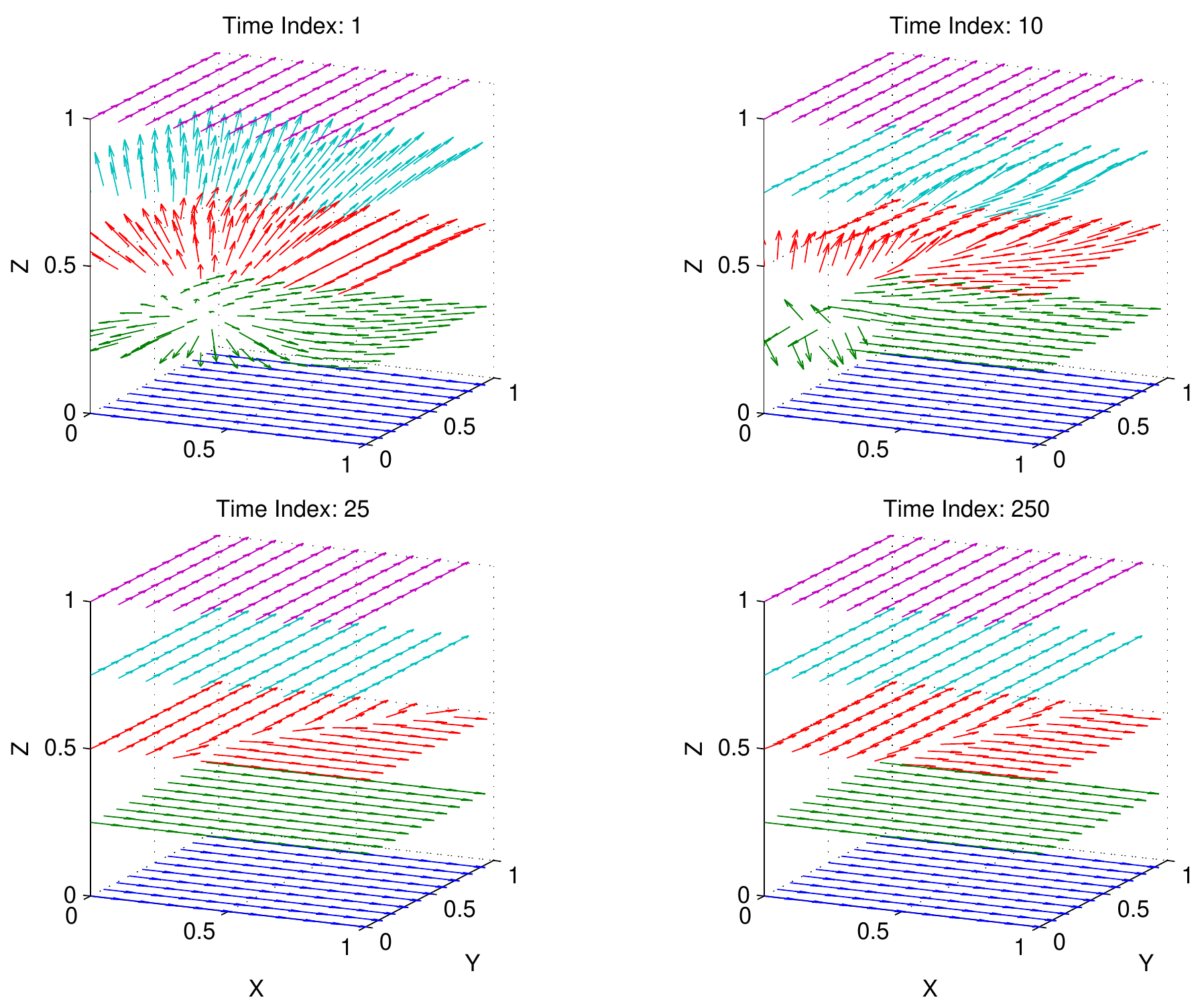}
\caption{Evolution toward an (equilibrium) plane defect (Section
\ref{sec:plane_defect_3D}).  The director field $\vn$ is shown at
five different horizontal slices.  The time step used was $\dt =
0.02$.}
\label{fig:Plane_Defect_3D_Director}
\end{center}
\end{figure}
Figure \ref{fig:Plane_Defect_3D_Director} shows the evolution of the director field $\vn$ toward the plane defect.
Only a few slices are shown in Figure \ref{fig:Plane_Defect_3D_Director} because of the simple form of the equilibrium solution.

Figure \ref{fig:Plane_Defect_3D_Slice_Director_And_Scalar} (left) shows the components of $\vn$ evaluated along a one dimensional vertical slice.  Clearly, the numerical solution approximates the exact solution well, except at the narrow transition region near $z = 0.5$.  Furthermore, Figure \ref{fig:Plane_Defect_3D_Slice_Director_And_Scalar} (right) shows the corresponding evolution of the degree of orientation parameter $s$ (evaluated along the same one dimensional vertical slice).  One can see the regularizing effect that $s$ has, i.e. at equilibrium, $s \approx 0.008$ at the $z=0.5$ plane (the defect plane of $\vn$).  Our numerical experiments suggest that $s|_{z=0.5} \rightarrow 0$ as the mesh size goes to zero.
\begin{figure}
\begin{center}

\psfrag{director}{\hspace{-0.6cm}\small{comp. of }$\vn$} \psfrag{scalar}{$s$}

\includegraphics[width=4.0in]{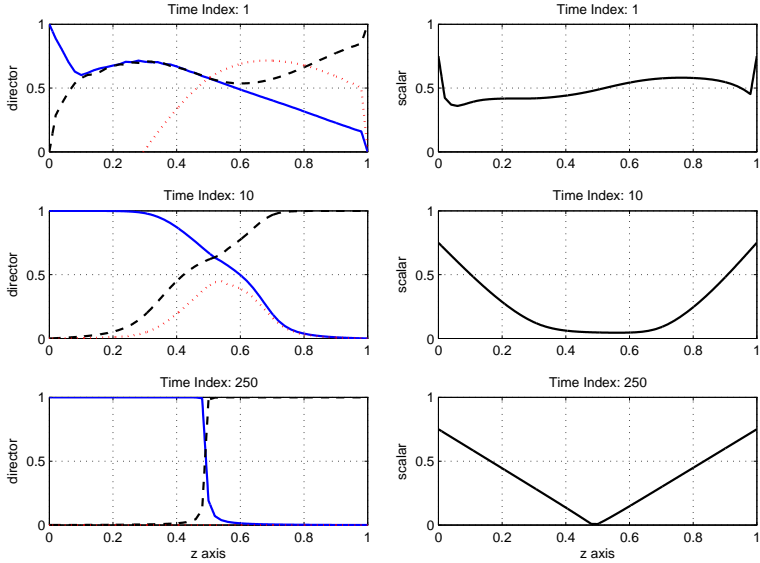}
\caption{Evolution toward an (equilibrium) plane defect (Section \ref{sec:plane_defect_3D}); time step is $\dt = 0.02$.  Left: plots of the three components of $\vn$, evaluated along the vertical line $x=0.5,y=0.5$, are shown at three time indices (solid blue curve: $\vn \cdot \ve_1$, dashed black curve: $\vn \cdot \ve_2$, dotted red curve: $\vn \cdot \ve_3$).  At equilibrium, $\vn$ is nearly piecewise constant with a narrow transition region around $z = 0.5$. Right: plots of the degree-of-orientation $s$, corresponding to $\vn$, are shown.  The equilibrium solution is piecewise linear, with a kink at $z=0.5$ where $s \approx 0.008$.}
\label{fig:Plane_Defect_3D_Slice_Director_And_Scalar}
\end{center}
\end{figure}

\subsection{Fluting effect and propeller defect}\label{sec:propeller_defect}

This example further investigates the effect of $\kappa$ on the presence of defects. An exact solution of a line defect in a right circular cylinder is given in \cite[Sec. 6.5]{Virga_book1994}.  They show that for $\kappa$ sufficiently large (say $\kappa > 1$) the director field is smooth, but if $\kappa$ is sufficiently small, then a line defect in $\vn$ appears along the axis of the cylinder.  Our numerical experiments confirm this.

To further illustrate this effect, we conducted a similar experiment
for a unit cube domain $\Om = (0, 1)^3$.  Again, for simplicity, we
\emph{remove} the double well potential.
We set Dirichlet boundary conditions for $(s,\vn)$
on the vertical sides of the cube
$\bdys=\bdyvn=\overline{\partial \Om} \cap ( \{ x
= 0\} \cup \{ x = 1\} \cup \{ y = 0\} \cup \{ y = 1\} )$,
with
\begin{equation}\label{eqn:bc_propeller_defect}
\begin{split}
  s = s^*, \qquad \vn(x,y,z) = \frac{(x,y) - (0.5,0.5)}{|(x,y) - (0.5,0.5)|}, \\
\end{split}
\end{equation}
and Neumann conditions
$\vnu \cdot \nabla s = 0$ and $\vnu \cdot \nabla \vn = 0$
on the top and bottom parts of $\partial \Om$;
this situation is not covered by Section \ref{sec:consistency}.
Figure \ref{fig:Fluting_Effect_3D_Horiz_Slices} shows the equilibrium solution when $\kappa = 2$.  The $z$-component of $\vn$ is \emph{not} zero, i.e. it points out of the plane of the horizontal slice that we plot.  This is referred to as the ``fluting effect'' (or escape to the third dimension \cite{Virga_book1994}).  In this case, the degree of orientation parameter $s$ is bounded well away from zero, so the director field is smooth (i.e. no defect).
\begin{figure}
\begin{center}

\includegraphics[width=5.0in]{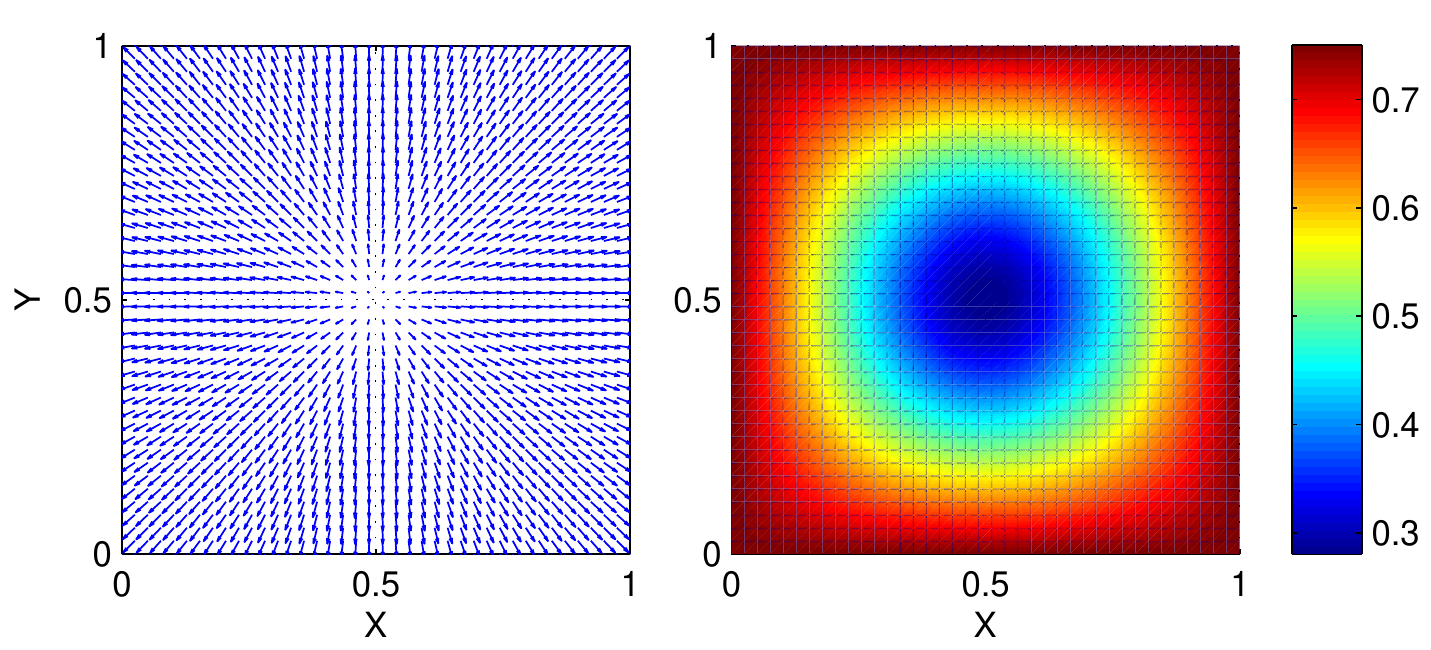}
\caption{Equilibrium state (Section \ref{sec:propeller_defect}) of $\vn$ and $s$.  One horizontal slice ($z = 0.5$) is plotted: $\vn$ on the left, $s$ on the right ($\vn$ and $s$ are approximately independent of $z$).  The director field points out of the plane (i.e. $\vn \cdot \ve_3 \neq 0$) and $s > 0.278$, so there is no defect.}
\label{fig:Fluting_Effect_3D_Horiz_Slices}
\end{center}
\end{figure}

Next, we choose $\kappa = 0.1$, and initialize our gradient flow scheme with $s = s^*$ and a regularized point defect away from the center of the cube for $\vn$.
\begin{figure}
\begin{center}


\includegraphics[width=4.5in]{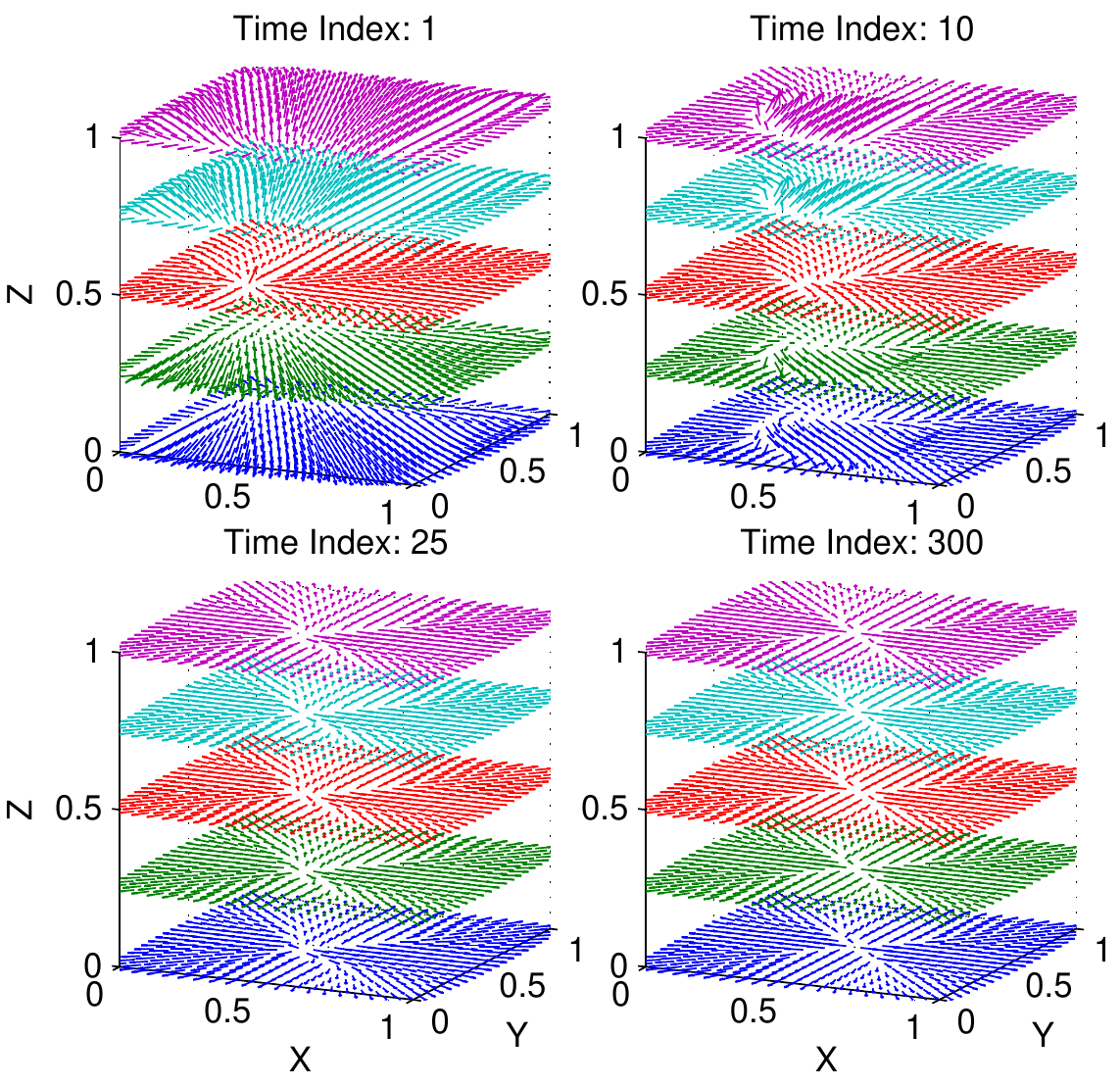}
\caption{Evolution toward an (equilibrium) ``propeller'' defect (Section \ref{sec:propeller_defect}).  Director field $\vn$ is shown at five different horizontal slices through the cube. The time step used was $\dt = 0.02$.}
\label{fig:Propeller_Defect_3D_Director}
\end{center}
\end{figure}
Figure \ref{fig:Propeller_Defect_3D_Director} shows the evolution of the director field $\vn$ toward a ``propeller'' defect (two plane defects intersecting).  Figure \ref{fig:Propeller_Defect_3D_Slice_Director_And_Scalar} shows $\vn$ and $s$ in their final equilibrium state at the $z = 0.5$ plane.  Both $\vn$ and $s$ are nearly uniform with respect to the $z$ variable.  The regularizing effect of $s$ is apparent, i.e. $s \approx 2 \times 10^{-5}$ near where $\vn$ has a discontinuity.  The 3-D shape of the defect resembles two planes intersecting near the $x=0.5$, $y=0.5$ vertical line, i.e. the defect looks like an ``X'' extruded in the $z$ direction.
\begin{figure}
\begin{center}


\includegraphics[width=5.0in]{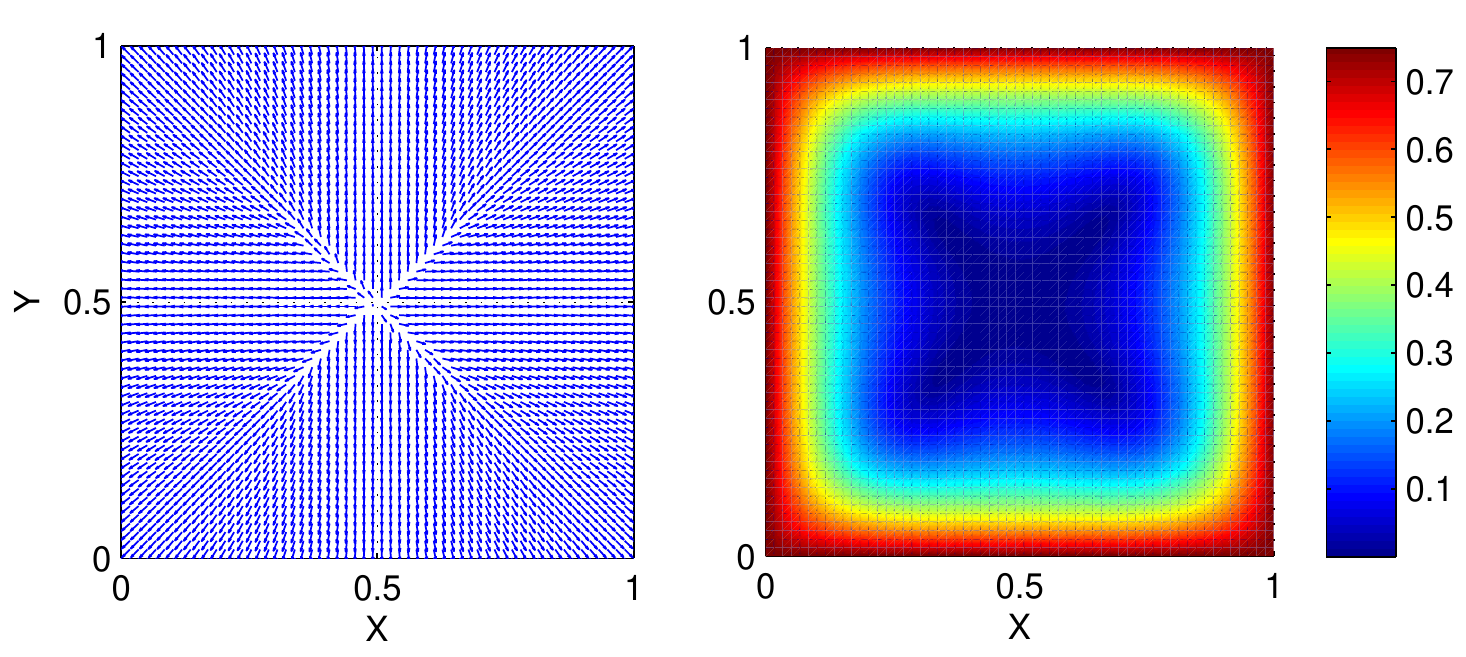}
\caption{Equilibrium state of a ``propeller'' defect (Section \ref{sec:propeller_defect}).  One horizontal slice ($z = 0.5$) is plotted: $\vn$ on the left, $s$ on the right ($\vn$ and $s$ are nearly independent of $z$).  The $z$-component of $\vn$ is zero and $s \approx 2 \times 10^{-5}$ near the discontinuity in $\vn$.}
\label{fig:Propeller_Defect_3D_Slice_Director_And_Scalar}
\end{center}
\end{figure}

\subsection{Floating plane defect}\label{sec:floating_plane_defect}

This example investigates the effect of the domain shape on the
defect.  The setup here is essentially the same as in Section
\ref{sec:propeller_defect}, with $\kappa = 0.1$, except the domain is
the rectangular box $\Om = (0, 1) \times (0, 0.7143) \times (0, 1)$.  Figure \ref{fig:Floating_Plane_Defect_3D_Slice_Director_And_Scalar} shows $\vn$ and $s$ in their final equilibrium state at the $z = 0.5$ plane.  Both $\vn$ and $s$ are approximately uniform with respect to the $z$ variable.  Instead of the propeller defect, we get a ``floating'' plane defect aligned with the major axis of the box.  Again, the regularizing effect of $s$ is apparent, i.e. $s \approx 7 \times 10^{-5}$ near where $\vn$ has a discontinuity.
\begin{figure}
\begin{center}


\includegraphics[width=5.0in]{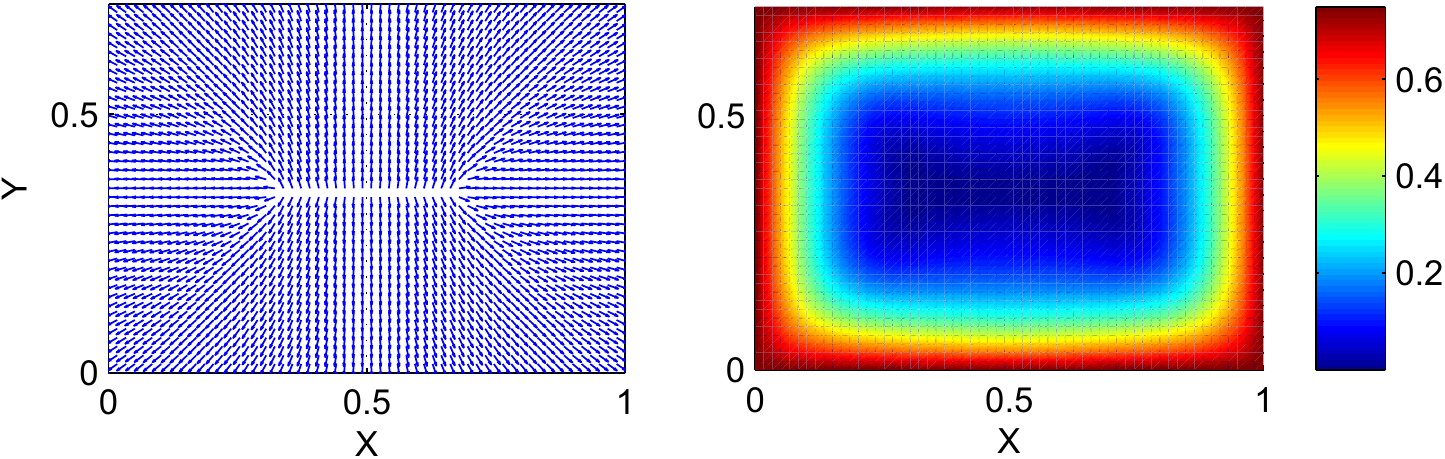}
\caption{Equilibrium state of a floating plane defect on a
rectangular domain (Section \ref{sec:floating_plane_defect}).  One horizontal slice ($z = 0.5$) is plotted: $\vn$ on the left, $s$ on the right ($\vn$ and $s$ are approximately independent of $z$).  The $z$-component of $\vn$ is zero and $s > 0$ with $s \approx 7 \times 10^{-5}$ near the discontinuity in $\vn$.}
\label{fig:Floating_Plane_Defect_3D_Slice_Director_And_Scalar}
\end{center}
\end{figure}

\section{Conclusion}\label{sec:conclusion}

We introduced and analyzed a robust finite element method for a degenerate energy functional that models nematic liquid crystals with variable degree of orientation.  We also developed a quasi-gradient flow scheme for computing energy minimizers, with a strict monotone energy decreasing property.  The numerical experiments show a variety of defect structures that Ericksen's model exhibits.  Some of the defect structures are high dimensional with surprising shapes (see Figure \ref{fig:Propeller_Defect_3D_Slice_Director_And_Scalar}).  We mention that \cite{Kralj_PRSA2014} also found a ``propeller'' (or ``X'') shaped defect within a two dimensional Landau-deGennes ($\vQ$-tensor) model.  An interesting extension of this work is to couple the effect of external fields (e.g. magnetic and electric fields) to the liquid crystal as a way to drive and manipulate the defect structures.

\medskip
\textbf{Acknowledgements:}
Nochetto and Zhang were partially supported by the NSF through the grant
DMS-1411808. Walker was partially supported by the NSF through
the grants DMS-1418994, DMS-1555222. Nochetto also acknowledges support of the
Institut Henri Poincar\'e (Paris)
and Zhang acknowledges support of the University of Maryland
by the Brin post-doctoral fellowship.
Finally, we thank L. Ambrosio for useful discussions regarding the
regularization argument in Proposition \ref{P:regularization}.\looseness=-1


\bibliographystyle{siam}
\bibliography{MasterBibTeX}

\end{document}